\documentclass[10pt,twoside,english,reqno,a4paper]{amsart}
\usepackage{listings,graphicx,amsmath,varioref,amscd,amssymb,color,bm,stmaryrd,amsthm,amsfonts,graphics,geometry,latexsym,pgf,pst-all}
\theoremstyle{plain}
\usepackage{esint}
\usepackage{cancel, ulem}
\usepackage{amsthm}

\theoremstyle{plain}
\newtheorem{theorem}{Theorem}[section]

\newtheorem{lemma}[theorem]{Lemma}

\newtheorem{corollary}[theorem]{Corollary}

\usepackage{geometry}
\usepackage[colorlinks=false]{hyperref}

\theoremstyle{definition}
\newtheorem{defin}[theorem]{Definition}

\newtheorem{remark}[theorem]{Remark}
\newtheorem{example}{Example}

\theoremstyle{remark}

\def\bk{\color{black}}

\numberwithin{equation}{section}

\def\into{\int_{\Omega}}
\def\dis{\displaystyle}

\DeclareMathOperator{\sgn}{sgn}
\DeclareMathOperator{\R}{\mathbb{R}}

\newcommand{\car}[1]{\raise1pt\hbox{$\chi$}_{#1}}

\def\rn{\mathbb{R}^N}
\def\re{\mathbb{R}}

\newcommand{\res}{\!\!\mathop{\hbox{
			\vrule height 7pt width .5pt depth 0pt
			\vrule height .5pt width 6pt depth 0pt}}
	\nolimits}

\makeatletter
\@namedef{subjclassname@2020}{%
  \textup{2020} Mathematics Subject Classification}
\makeatother
\begin{document}
\title[Absorption terms in  Dirichlet problems for the prescribed mean curvature equation]{The role of absorption terms in  Dirichlet problems for the prescribed mean curvature equation}

\author[F. Oliva]{Francescantonio Oliva}
\author[F. Petitta]{Francesco Petitta}
\author[S. Segura de Le\'on]{Sergio Segura de Le\'on}

\address{Francescantonio Oliva
	\hfill \break\indent
Dipartimento di Scienze di Base e Applicate per l' Ingegneria, Sapienza Universit\`a di Roma
	\hfill \break\indent
	Via Scarpa 16, 00161 Roma, Italy}
\email{\tt francescantonio.oliva@uniroma1.it}
\address{Francesco Petitta 	\hfill \break\indent Dipartimento di Scienze di Base e Applicate per l' Ingegneria, Sapienza Universit\`a di Roma
		\hfill \break\indent
		Via Scarpa 16, 00161 Roma, Italy}
\email{\tt francesco.petitta@uniroma1.it}
\address{Sergio Segura de Le\'on
	\hfill \break\indent Departament d'An\`alisi Matem\`atica,
	Universitat de Val\`encia, \hfill\break\indent Dr. Moliner 50,
	46100 Burjassot, Val\`encia, Spain.} \email{{\tt
		sergio.segura@uv.es}}


\begin{abstract}
In this paper we study existence and uniqueness of solutions to Dirichlet problems as 
$$
	\begin{cases}
		g(u) \dis -\operatorname{div}\left(\frac{D u}{\sqrt{1+|D u|^2}}\right) = f & \text{in}\;\Omega,\\
		u=0 & \text{on}\;\partial\Omega,
	\end{cases}
$$
where $\Omega$  is an open bounded subset of $\R^N$ ($N\ge 2$) with Lipschitz boundary, $g:\re\to\re$ is a continuous function and $f$ belongs to some Lebesgue spaces. In particular, under suitable saturation and sign assumptions,  we explore the regularizing effect given by the absorption term $g(u)$ in order to get a  solutions for data $f$ merely belonging to $L^1(\Omega)$ and with no smallness assumptions on the norm. We also prove a  sharp boundedness result  for data in $L^{N}(\Omega)$ \bk  as well as uniqueness if $g$ is increasing.
\end{abstract}

\maketitle
\tableofcontents

\section{Introduction}

We study existence and uniqueness of solutions to problem
\begin{equation}
	\label{pbintro}
	\begin{cases}
		g(u) \dis -\operatorname{div}\left(\frac{D u}{\sqrt{1+|D u|^2}}\right) = f & \text{in}\;\Omega,\\
		u=0 & \text{on}\;\partial\Omega,
	\end{cases}
\end{equation}
where $\Omega$  is an open bounded subset of $\R^N$ ($N\ge 2$) with Lipschitz boundary, $g:\re\to\re$ is a continuous, and  the datum $f$ belongs  to $L^1(\Omega)$.

The main purpose of this paper is to  describe the regularizing effect of zero order absorption terms on the existence of solutions for  boundary value problems as in \eqref{pbintro}. \bk

Let us recall that equation in \eqref{pbintro}, if $f=0$ and without any absorption,  falls in the well known case of minimal surface equation
$$
\dis \operatorname{div}\left(\frac{D u}{\sqrt{1+|D u|^2}}\right) =0,
$$
   the name deriving from the fact that, for a smooth  function $u$, the involved operator calculates the mean curvature of the graph of $u$ at each point $(x,u(x))$; due to this fact such an operator is also called {\it non-parametric mean curvature operator}. 

Several cases  of  (non-parametric) prescribed mean curvature equation of the type
\begin{equation}
\label{pb0}
\begin{cases}
\dis -\operatorname{div}\left(\frac{D u}{\sqrt{1+|D u|^2}}\right) = f & \text{in}\;\Omega,\\
u=0 & \text{on}\;\partial\Omega,
\end{cases}
\end{equation}
have been considered as well  in  literature    starting by \cite{se}, \cite{g76,g78}, \cite{g}, and   \cite{fg84} to present a non-complete list. 

It is also worth to point out  that the  equation in \eqref{pbintro}, with $g(s)=s$, corresponds to the resolvent equation of the following evolution equation

\begin{equation}\label{acm} u_t  =\operatorname{div}\left(\frac{D u}{\sqrt{1+|D u|^2}}\right) \,;\end{equation}
 roughly speaking, proving existence and uniqueness for \eqref{pbintro},  can be considered as a first step in order to apply Crandall-Liggett theory (\cite{cl}) to look for   mild solutions to the corresponding evolution problem. In a more general context these type of arguments  have been successfully applied in order to get existence and uniqueness of Cauchy initial-boundary value problems involving equations as in  \eqref{acm}  in the framework of entropy type solutions and with $L^1$-initial data (see \cite{dt, ACM_MA, ACM_RMI, ACM, gm} for a quite exaustive account on this issue).\bk 
 \medskip

Concerning less theoretical issues, problems  as in \eqref{pb0} arise     in the study of  combustible gas dynamics (see \cite{yy} and references therein) as well as in  surfaces capillary problem as pendant liquid drops (\cite{finn74,cf1,cf2, fg84}) and also in design of water-walking devices (\cite{ww}, see also \cite{lc}). 

\medskip

Prescribed mean curvature equations as in \eqref{pb0}  formally represent the Euler-Lagrange equation of the functional 
$$
\mathcal{J}(v)=\int_\Omega \sqrt{1+|\nabla v|^2}\,dx  + \int_{\partial\Omega} |v| d\mathcal{H}^{N-1}  - \into fv  dx\,, 
$$
involving the area functional. 

As regards the solvability of problems as in \eqref{pb0}, a smallness assumption on the data naturally appears: indeed, if we  formally integrate the equation in \eqref{pb0} in a smooth sub-domain of $A\subset\Omega$, an application of 
 the divergence theorem gives  the following {\it necessary condition}
$$
\left|\int_A f(x) dx\right| =\left| \int_{\partial A} \frac{D u}{\sqrt{1+|D u|^2}}\cdot \nu_A ds \right|< {\rm Per}(A)
$$
where ${\rm Per}(A)$ indicates the perimeter of $A$ and $\nu_A$ is the outer normal unit vector. 
That is,  some sort of  smallness assumption on the datum $f$ is needed in order to get existence in problems as \eqref{pb0}. This is a typical feature of problems arising from functionals with linear growth as, for instance,  the one driven by the  $1$-laplacian  (see for instance \cite{CT,KS, DGOP}). See also Remark \ref{rem:1} below for more details on this structural obstruction. In \cite{g} M. Giaquinta  shows the unique solvability  of \eqref{pb0},  in a variational sense,  \bk in the space of functions with bounded variation provided $f$ is measurable and there exists $\varepsilon_0>0$ such that for every smooth $A\subseteq \Omega$
\begin{equation}\label{giac}
\left|\int_A f(x) \,dx\right|  \leq (1-\varepsilon_0){\rm Per}(A)\,. 
\end{equation}

In \cite{g76} it is shown that 
$$
||f||_{L^N (\Omega)}< N\omega_{N}^{\frac1N} \,,
$$
is a general condition under which \eqref{giac} holds, where $\omega_N$ is the measure of the unit ball of $\rn$ and it is a sharp request in order to get  bounded solutions for problem \eqref{pb0} (see \cite{gop2}). 

\medskip

Less regularity for data $f\in L^q(\Omega)$  below the threshold $q=N$ is known to be more challenging  than  the classical variational  setting of $BV$-solutions for equations arising from functionals with linear growth and one needs a different approach, see for instance \cite{mst2} and \cite{lops}. We also point out that  these generalized solutions are,   in general,  bounded only for data $f \in L^{N,\infty} (\Omega)$ with small norm.  

\medskip 

As we said, our  main focus \bk  consists in  analyze the regularizing effect of zero order absorption terms for problems  as   in \eqref{pbintro}
\medskip 
where $g:\re \to \re$ is a continuous functions such that  $$g(s) \to \pm \infty\ \ \text{ as}\ \   s\to\pm\infty \ \ \ \   \text{and}\  \ g(s)s\geq 0.$$  
 We show that solutions of  \eqref{pb0}  do  exist for general data $f\in L^1(\Omega) $ no matter of the size of $f$ and,  if $g:\re\to\re$ is increasing,  the solution \bk is unique. Moreover,  if $f\in L^N(\Omega)$, then solutions to  \eqref{pbintro} lie in $L^{\infty}(\Omega)$, again, without any restriction on the norm of $f$. As a remarkable fact this result is sharp at Lorentz scale since, as  we will  show by means of an explicit  counter-example,  unbounded solutions may exist for data in $f\in L^{N,\infty}(\Omega)$.  The boundedness of solutions for $L^N$-data is a bit unexpected since the extension of the Calderon-Zygmund regularity theory just guarantees bounded solutions when data belong to $L^m(\Omega)$ for $m>N$.

\medskip
In the first part we work by approximation proving existence of a $BV$-solution of problem \eqref{pbintro} when $f\in L^N(\Omega)$; in this case the regular approximation scheme is suitably chosen involving   $p$-Laplacian type operators. This part has some overlap with \cite{dt} or \cite{ACM_RMI} in case $g(s)=s$.  We present our results    in case of a generic nonlinearity  $g$ and with a quite different approach based on the $L^{\infty}$-estimate.  
In the second part we look for  infinite energy solutions of problem \eqref{pbintro} when $f$ is a merely integrable function. In this case the approximation scheme is given by solutions to problems as \eqref{pbintro} whose existence has been proven in  the first part and we only approximate the datum $f$. We  remark that, from a different point of view, problem \eqref{pb0} (again in the case of linear absorption) with $L^1$-data   is studied in \cite{ACM_MA} and \cite{gm}.

\medskip \medskip 

The plan of the paper is the following: in Section \ref{due} we set the basic machinery on $BV$ spaces (the natural space in which these problems are well settled), and the 
Anzellotti-Chen-Frid theory of pairings between bounded  vector fields whose divergence lies in some Lebesgue spaces  and gradients of $BV$ functions. Section \ref{secL2} is devoted to present the existence and uniqueness theory of finite energy solutions to problem \eqref{pbintro}  in case of data $f\in L^N(\Omega)$.  
The core of the paper is the content of Section \ref{secL1} in which we prove existence and uniqueness of infinite energy  solutions to \eqref{pbintro} in full generality. In Section \ref{boun} we discuss the existence of finite energy solutions to \eqref{pbintro} when $f$ does not necessarily belong to $L^N(\Omega)$. In particular, if  $f\in L^{N,\infty}(\Omega)$, a smallness assumption is needed to guarantee the boundedness of the solutions; this hypothesis turns out to be sharp as shown by an explicit  example. 

\color{black}

\section{Notation and preparatory tools}  \label{due}

From here on  $\Omega$  will always represent  an open bounded subset of $\R^N$ ($N\ge 2$) with Lipschitz boundary.
We denote by $\mathcal H^{N-1}(E)$ the $(N - 1)$-dimensional Hausdorff measure of a set $E$,  while $|A|$ stands for the $N$-dimensional  Lebesgue measure $\mathcal{L}^N$ of a set $A\subset \rn$. We denote by $\chi_{A}$ the characteristic function of a set $A\subset \rn$.

By $\mathcal{M}(\Omega)$ we indicate the space of Radon measures with finite total variation over $\Omega$ and we will call mutually singular (or mutually orthogonal) two Radon measures $\mu$ and $\nu$ in $\mathcal{M}(\Omega)$ such that there exists a measurable set $A\subset \Omega$ satisfying 
$$
\mu\res A=\nu\res (\Omega\backslash A)=0. 
$$ 

For a fixed $k>0$, we use the truncation functions $T_{k}:\R\to\R$ and $G_{k}:\R\to\R$ defined, respectively,  by
\begin{align*}
T_k(s):=&\max (-k,\min (s,k))\ \ \text{\rm and} \ \ G_k(s):=s- T_k(s).
\end{align*}

If no otherwise specified, we denote by $C$ several positive constants whose value may change from line to line and, sometimes, on the same line. These values will only depend on the data but they will never depend on the indexes of the sequences we will gradually introduce. Let us explicitly mention that we will not relabel an extracted compact subsequence.

For simplicity's sake, and if there is no ambiguity,  we will often use the following notation:
$$
 \int_\Omega f:=\int_\Omega f(x)  dx , 
$$
{and, if $\mu$ is a Radon measure,
\[\int_\Omega f\mu:=\int_\Omega f\, d\mu\,.\]}

Finally, we will denote by ${\rm sgn} (s)$ the multi-valued sign function defined by 
$$
\sgn (s):=\begin{cases}
1 & \text{if }\ s>0\\
[-1,1] & \text{if }\ s=0\\
-1 & \text{if }\ s<0.
\end{cases}
$$

\subsection{$BV$ spaces and the area functional}
We refer to \cite{afp} for a complete account on $BV$-spaces.

Let
$$BV(\Omega):=\{ u\in L^1(\Omega) : Du \in \mathcal{M}(\Omega)^N \}.$$
By $Du \in \mathcal{M}(\Omega)^N$ we mean that  each distributional partial derivative of $u$ is a Radon measure with finite total variation. Then the total variation of $D u$ is given by
$$\displaystyle |D u| = \sup\left\{\int_\Omega u\sum_{i=1}^{N}  \operatorname{\frac{\partial \phi_i}{\partial  x_i}}, \ \phi_i \in C^1_0(\Omega, \mathbb{R}), \ |\phi_i|\le 1, \forall i=1,...,N\right\}.$$

Each $u\in BV(\Omega)$ exhibits a trace on $\partial\Omega$ which belongs to $L^1(\partial\Omega)$. Henceforth we will use the same notation for a $BV$--function and its trace.

We underline that $BV(\Omega)$ endowed with the norm
$$\displaystyle ||u||_{BV(\Omega)}=\int_{\partial\Omega}
|u|\, d\mathcal H^{N-1}+ \int_\Omega|Du|,$$
is a Banach space.

A Radon measure $\mu$ can be uniquely decomposed as
$\mu=\mu^a+\mu^s$
where $\mu^a$ is absolutely continuous with respect to the Lebesgue measure $\mathcal{L}^N$ while $\mu^s$ is concentrated on a set of zero Lebesgue measure, i.e. $\mu^a$ and $\mu^s$ are mutually singular.

\medskip

If $u\in BV(\Omega)$ the measure $\sqrt{1+|Du|^2}$ is defined as
$$\displaystyle  \sqrt{1+ | D u|^2} (E)= \sup\left\{\int_E \phi_{N+1} - \int_E u\sum_{i=1}^{N}  \operatorname{\frac{\partial \phi_i}{\partial x_i}}, \ \phi_i \in C^1_0(\Omega, \mathbb{R}), \ |\phi_i|\le 1, \forall i =1,...,N+1\right\}\,,$$
for any Borel set $E\subseteq \Omega$.
The notation
$$\int_\Omega \sqrt{1+ | D u|^2}$$
stands for  the total variation of the  $\mathbb{R}^{N+1}$-valued measure which formally represents $(\mathcal{L}^N, Du)$. Notice that, if $u$ is smooth,  then
$$|(\mathcal{L}^N, \nabla u)| (\Omega)=\int_\Omega \sqrt{1+ |\nabla u|^2}$$
gives the area of the graph of $u$.
 Let us \bk also observe that it follows from the decomposition in absolutely continuous and singular part with respect to the Lebesgue measure that one has
\begin{equation*}
	\sqrt{1+ | D u|^2}= \sqrt{1+ | D^a u|^2}\mathcal{L}^N+|D^s u| \,,
\end{equation*}
 where we use the following notations $D^a u:=(D u)^a$ and $D^s u:=(D u)^s$. \bk

In what follows we will use the following semicontinuity classical results; firstly, the  functional $$J_1(v)=\int_\Omega \sqrt{1+ | D v|^2}\varphi + \int_{\partial\Omega} |v|\varphi \, d\mathcal H^{N-1} , \ \ \text{for all} \ 0\le \varphi \in C^1(\overline{\Omega})$$
is lower semicontinuous in $BV(\Omega)$ with respect to the $L^1(\Omega)$ convergence. On the other hand  the  functional \begin{equation}\label{j2} J_2(v) = \int_\Omega \sqrt{1-|v|^2}\varphi \ \ \text{for all} \ 0\le \varphi \in C^1(\overline{\Omega})\end{equation} defined on functions  $|v|\leq 1$ is weakly upper semicontinuous in $L^1(\Omega)$  (see Corollary $3.9$ of \cite{brezis}).\bk

\subsection{The Anzellotti-Chen-Frid theory} Let us briefly present the  $L^\infty$-vector fields  theory due to \cite{An} and \cite{CF} in the case   of bounded fields $z$ whose divergence is in $L^q(\Omega)$.

Let $q\ge 1$ and 
$$X(\Omega)_q:=\{ z\in L^\infty(\Omega)^N : \operatorname{div}z \in L^q(\Omega) \}.$$

In \cite{An}, under suitable compatibility conditions that we shall outline later,  given a function $v\in BV(\Omega)$ and a bounded vector field $z\in X(\Omega)_q$, the following distribution $(z,Dv): C^1_c(\Omega)\to \mathbb{R}$ is defined:
\begin{equation}\label{dist1}
\langle(z,Dv),\varphi\rangle:=-\int_\Omega v\varphi\operatorname{div}z-\int_\Omega
vz\cdot\nabla\varphi,\quad \varphi\in C_c^1(\Omega)\,.
\end{equation}

Let us stress that  \eqref{dist1} is well defined provided one of the following {\it compatibility conditions} hold:  
\begin{equation}\label{cc1} v \in BV(\Omega)\  \text{and}\ \  z\in X(\Omega)_N\,,\end{equation} 

or  
\begin{equation}\label{cc3} v \in BV(\Omega)\cap L^\infty(\Omega)\  \text{and}\  z\in X(\Omega)_1\,.\end{equation} 

We point out that an admissible compatibility condition is also  $v \in BV(\Omega)$ and 
$\operatorname{div}z\in L^{N,\infty}(\Omega)$, where $L^{N,\infty}(\Omega)$ is the usual Lorentz space (see \cite{PKF} for an introduction on such function spaces) also known as Marcinkiewicz space of exponent $N$. \bk

\medskip 
Moreover, it holds
\begin{equation*}\label{finitetotal}
|\langle   (z, Dv), \varphi\rangle| \le ||\varphi||_{L^{\infty}(U) } ||z||_{L^\infty(U)^N} \int_{U} |Dv|\,,
\end{equation*}
for all open set $U \subset\subset \Omega$ and for all $\varphi\in C_c^1(U)$, and 
\begin{equation*}\label{finitetotal1}
\left| \int_B (z, Dv) \right|  \le  \int_B \left|(z, Dv)\right| \le  ||z||_{L^\infty(U)^N} \int_{B} |Dv|\,,
\end{equation*}
for all Borel sets $B$ and for all open sets $U$ such that $B\subset U  \subset \subset \Omega$.
Every $z \in X(\Omega)_q$ has a weak trace on $\partial \Omega$ of its normal component which is denoted by
$[z, \nu]$, where $\nu(x)$ is the outward normal unit vector defined for $\mathcal H^{N-1}$-almost every $x\in\partial\Omega$ (see  \cite{An}), such that
\begin{equation*}\label{des1}
||[z,\nu]||_{L^\infty(\partial\Omega)}\le ||z||_{L^\infty(\Omega)^N}\,.
\end{equation*}

The following Green formula holds (see \cite[Theorem 1.9]{An}):
\begin{lemma}\label{21}
Let $z \in L^{\infty}(\Omega)^N$ and $v\in BV(\Omega)$, then it holds
	\begin{equation}\label{green}
		\int_{\Omega} v \operatorname{div}z + \int_{\Omega} (z, Dv) = \int_{\partial \Omega} v[z, \nu] \ d\mathcal H^{N-1}\,,
	\end{equation}
	provided one of the compatibility conditions \eqref{cc1}  or \bk \eqref{cc3} is in force. 
\end{lemma}	
Let us recall the following technical result due again to \cite[Theorem 2.4]{An}. \begin{lemma}\label{lemanzas}
Let $z \in L^{\infty}(\Omega)^N$ and $v\in BV(\Omega)$, then it holds
$$(z, D u)^a = z \cdot D^a u.$$
provided one of the compatibility conditions \eqref{cc1}  or \bk \eqref{cc3} is in force. 

\end{lemma}

{\subsection{An algebraic inequality} In what follows we will have to apply an algebraic inequality, which is next set for the sake of completeness.
If $a\ge 0$ and $0\le b\le 1$, then
\begin{equation}\label{ine:1}
  ab\le \sqrt{1+a^2}-\sqrt{1-b^2}.
\end{equation}
To check it, just realize that writing as
\[ab+\sqrt{1-b^2}\le \sqrt{1+a^2}\]
squaring and simplifying, we get \eqref{ine:1} is equivalent to
\[a^2(1-b^2)-2ab\sqrt{1-b^2}+b^2\ge0.\]
This inequality holds since the left-hand side is a square.
As a consequence of \eqref{ine:1} and the Cauchy-Schwarz inequality, we deduce that if $A, B\in\R^N$ with $|B|\le 1$, then
\begin{equation}\label{ine:2}
  A\cdot B\le \sqrt{1+|A|^2}-\sqrt{1-|B|^2}.
\end{equation}
}

\section{$BV$-solutions in presence of $L^N$-data}  \label{secL2}

In this section we deal with the following problem:
\begin{equation}
\label{pb}
\begin{cases}
g(u)\dis -\operatorname{div}\left(\frac{D u}{\sqrt{1+|D u|^2}}\right) = f & \text{in}\;\Omega,\\
u=0 & \text{on}\;\partial\Omega,
\end{cases}
\end{equation}
where $f$ belongs  to $L^N(\Omega)$ and $g$ is a continuous function such that
\begin{equation}\label{condg} \lim_{s\to \pm \infty} g(s) = \pm\infty \ \ \text{and}\  \ g(s)s\geq 0\  \ s\in \re.\end{equation}

\medskip 
Let us start by  specifying  what we mean by a solution to \eqref{pb}.
\begin{defin}
	\label{weakdef} 
	Let $f\in L^N(\Omega)$. A function {$u\in BV(\Omega)$} is a solution to problem \eqref{pb} if there exists $z\in X(\Omega)_N$ with $||z||_{L^\infty(\Omega)^N}\le 1$ such that 
	\begin{align}
		&g(u)-\operatorname{div}z = f \ \ \text{in}\ \ \mathcal{D}'(\Omega), \label{def_distrp=1}
		\\
		&(z,Du)=\sqrt{1+|Du|^2} - \sqrt{1-|z|^2} \label{def_zp=1} \ \ \ \ \text{as measures in } \Omega,
		\\
		&u(\sgn{u} + [z,\nu])(x)=0 \label{def_bordop=1}\ \ \ \text{for  $\mathcal{H}^{N-1}$-a.e. } x \in \partial\Omega.
	\end{align}
\end{defin}

\begin{remark}\label{remdef}
	Let us underline that \eqref{def_zp=1} aims to give an interpretation to $D u/\sqrt{1+|D u|^2}$.
	 Indeed, if $u$ is smooth and $z=\frac{\nabla u}{\sqrt{1+|\nabla u|^2}}$, then one has
	$$
	(z, \nabla u)= z\cdot \nabla u=\frac{|\nabla u|^2}{\sqrt{1+|\nabla u|^2}} \,,
	$$
	which, after a simple manipulation, gives the right-hand of \eqref{def_zp=1}.
	
	It is also worth mentioning that, under the assumptions of Lemma \ref{lemanzas}, \eqref{def_zp=1} turns out to be equivalent to  require that both
	\begin{equation}\label{remdef1}
		z\cdot D^a u = \sqrt{1+ |D^a u|^2} - \sqrt{1-|z|^2}
	\end{equation}
	and
	\begin{equation*} \label{remdef2}
		(z,Du)^s = |D^s u|,
	\end{equation*}
		holds (see \cite{dt}). 
		
		We stress that  condition \eqref{def_zp=1} 
		was leveraged in \cite[Theorem 3.1]{dt} in order to  characterize  the subdifferential of the functional 
		 $$u\mapsto \int_{\Omega}\sqrt{1+|Du|^2}+ \int_{\partial\Omega}|u|\, d\mathcal H^{N-1} -  \into fu ,$$ 
		 associated to \eqref{pb} with $g\equiv0$.

	Let us also stress that, once \eqref{remdef1} is in force, $z$ is uniquely defined by	
	\begin{equation}\label{zesplicito}
		z= \frac{D^a u}{\sqrt{1+|D^a u |^2}}.
	\end{equation}
	 This is a striking  difference with some others  flux-limited diffusion operators as the $1$-laplacian or the transparent media one \cite{ACM, ABCM, GMP}).

	With regard to \eqref{def_bordop=1},  it is nowadays the classical way the Dirichlet datum is meant for these type of equations as, in general, the  trace of the solutions is not attained pointwise. Roughly speaking, it means that  at any  point of $\partial\Omega$ either $u$ is zero  or the modulus of the weak trace of the normal component of $z$ is highest possible at the boundary. 
	
	Let us also observe that it follows from \eqref{def_distrp=1} that $g(u)\in L^N(\Omega)$.

	We conclude by stressing that, by standard embedding in $L^{\frac{N}{N-1}}(\Omega)$,   \eqref{def_distrp=1} also holds tested with functions in $BV(\Omega)$  as much as in $\mathcal{D}(\Omega)$. \bk 
\end{remark}

Let us state the existence result of this section:
\begin{theorem}\label{teomain}
	Let $f\in L^{N}(\Omega)$ and let $g$ satisfy \eqref{condg}. Then there exists a bounded solution $u$ to problem \eqref{pb} in the sense of Definition \ref{weakdef}. \end{theorem}

\begin{remark}\label{rem:1}
	Let us stress again  that, in absence of the absorption zero order term, existence of $BV$-solutions are expected only for small $f$'s belonging to $L^N(\Omega)$.
	To check that a smallness condition is needed in this case, assume that there exists a solution of problem \eqref{pb} without the absorption term and let $z$ be the associated vector field. Then, for every $v\in W_0^{1,1}(\Omega)$, Green's formula implies
	\[\left|\int_\Omega fv\right|=\left|\int_\Omega z\cdot \nabla v\right|\le \int_\Omega |\nabla v|\,.\]
	Thus $f\in W^{-1,\infty}(\Omega)$, the dual space of $W_0^{1,1}(\Omega)$, and $\| f\|_{W^{-1,\infty}(\Omega)}\le 1$.
	Theorem \ref{teomain} consequently shows that when dealing with the regularizing absorption term one gains that a solution always exists, and it belongs to $BV(\Omega)$, avoiding any small condition on the size of $f$.

\medskip

It is also worth noting that the hypothesis \eqref{condg} is necessary to get existence of a solution for every $f\in L^N(\Omega)$. Indeed, assume that $g$ is bounded from above, that is, there exists $M>0$ such that $g(s)\le M$ for all $s\in\R$ and let $f\in L^N(\Omega)$ with $f\ge0$. If $u$ is a solution to problem \eqref{pb} and define $f_1=f-g(u)$, then 
\[	\begin{cases}
	\dis -\operatorname{div}\left(\frac{D u}{\sqrt{1+|D u|^2}}\right) = f_1 & \text{in}\;\Omega,\\
		u=0 & \text{on}\;\partial\Omega,
	\end{cases}\]
	So, for what we said before,  we infer that $f_1\in W^{-1,\infty}(\Omega)$ and $\| f_1\|_{W^{-1,\infty}(\Omega)}\le1$. To compute $\| f\|_{W^{-1,\infty}(\Omega)}$, consider $v\in W_0^{1,1}(\Omega)$. It follows from $0\le f=f_1+g(u)\le f_1+M$ that 
	\[
	\left|\int_\Omega fv\right|\le \int_\Omega f|v|\le \int_\Omega f_1|v|+M\int_\Omega |v|
	\le \| f_1\|_{W^{-1,\infty}(\Omega)}\int_\Omega\big|\nabla v\big|+M|\Omega|^{1/N}\| v\|_{L^{\frac{N}{N-1}}(\Omega)}
	\]
	Appealing now to Sobolev's inequality, we obtain
	\[\left|\int_\Omega fv\right|\le\big(1+M|\Omega|^{1/N}\mathcal{S}_1\big)\int_\Omega|\nabla v|\]
	and consequently $\| f\|_{W^{-1,\infty}(\Omega)}\le 1+M|\Omega|^{1/N}\mathcal{S}_1$. Therefore, the datum $f$ cannot be arbitrary.
	
	An analogous argument can be developed assuming that  $g$ is bounded from below.
\end{remark}

By appealing  to the presence of the regularizing zero order term, we show that the $BV$-solution of \eqref{pb} is unique provided $g$ is increasing. 
\begin{theorem}\label{teomainuniqueL2}
	Let $g$ be an increasing function. Then there is at most one solution to problem \eqref{pb} in the sense of Definition \ref{weakdef}.
\end{theorem}

\subsection{Existence of finite energy solutions}

In order to prove Theorem \ref{teomain}, for $p>1$, we consider the following scheme of approximation:
\begin{equation}
	\label{pbp}
	\begin{cases}
		g_p(u_p)\dis \dis -\operatorname{div}\left(\frac{\nabla u_p}{\sqrt{1+|\nabla  u_p|^2}}\right) - (p-1)\operatorname{div}\left(|\nabla u_p|^{p-2}\nabla u_p\right)  = f_p & \text{in}\;\Omega,\\
		u_p=0 & \text{on}\;\partial\Omega,
	\end{cases}
\end{equation}
where $g_p(s)=T_{\frac{1}{p-1}}(g(s))$ for any $s\in \mathbb{R}$,  and $f_p= T_{\frac{1}{p-1}}(f)$. The existence of $u_p \in W^{1,p}_0(\Omega)\cap L^\infty(\Omega)$ such that
\begin{equation}\label{pbpw}
\int_\Omega g_p(u_p)v + \into \frac{\nabla u_p}{\sqrt{1+|\nabla  u_p|^2}}\cdot \nabla v + (p-1)\into |\nabla u_p|^{p-2}\nabla u_p\cdot \nabla v = \into f_pv\,, \ \ \forall v\in W^{1,p}_0(\Omega)
\end{equation}
follows by  standard monotonicity arguments (\cite{ll}).

\medskip 

We start proving that $u_p$ is bounded uniformly with respect to $p$ by appealing to an idea in \cite{DS}.

\begin{lemma}\label{lem_bounded}
	Let $f\in L^N(\Omega)$, let $g$ satisfy \eqref{condg},  and let $u_p$ be a solution to \eqref{pbp}. Then 
	\begin{equation}\label{bounded}
	\|u_p\|_{L^{\infty}(\Omega)}\leq C\,
	\end{equation}
	for some  positive constant $C$ not depending on $p$. 
\end{lemma}
\begin{proof}
	We take $G_k(u_p)$ where $k>0$ as a test function in \eqref{pbpw}, yielding to
	\begin{equation}\label{bounded1}
		\int_\Omega g_p(u_p)G_k(u_p) + \into \frac{|\nabla G_k(u_p)|^2}{\sqrt{1+|\nabla  G_k(u_p)|^2}} \le \int_\Omega f_p G_k(u_p).
	\end{equation}	
	
	For the first term on the left-hand of \eqref{bounded1} we have
	\begin{equation}\label{bounded2tris}
		\inf_{|s|\in [k,\infty)} |g_p(s)| \int_\Omega |G_k(u_p)| \le  \int_\Omega g_p(u_p)G_k(u_p),
	\end{equation}
		while for the second term on the left-hand of \eqref{bounded1} one has
	\begin{equation}\label{bounded2bis}
		\begin{aligned}
		 \int_\Omega \frac{|\nabla G_k(u_p)|^2}{\sqrt{1+|\nabla  G_k(u_p)|^2}} &= \int_{\{|u_p|>k\}} \sqrt{1+|\nabla  G_k(u_p)|^2} - \int_{\{|u_p|>k\}} \frac{1}{\sqrt{1+|\nabla  G_k(u_p)|^2}} 
		 \\
		 &\ge \int_\Omega |\nabla  G_k(u_p)| - |\{|u_p|>k\}|.
		 \end{aligned}
	\end{equation}
	
	For the right-hand of \eqref{bounded1} we write
	\begin{equation}\label{bounded2}
		\begin{aligned}
			\int_\Omega f_p G_k(u_p) & \le  \int_{\{|f|\le h\}} |f| |G_k(u_p)| + \int_{\{|f|> h\}}|f| |G_k(u_p)|
			\\
			&\le h\int_\Omega |G_k(u_p)| + \|f\chi_{\{|f|> h\}}\|_{L^N(\Omega)} \left(\int_{\Omega} |G_k(u_p)|^{\frac{N}{N-1}}\right)^{\frac{N-1}{N}}
			\\
			&\le h\int_\Omega |G_k(u_p)| + \|f\chi_{\{|f|> h\}}\|_{L^N(\Omega)} \mathcal{S}_1 \int_{\Omega} |\nabla G_k(u_p)|,
		\end{aligned}
	\end{equation}	
	after applications of the H\"older and Sobolev inequalities (here $\mathcal{S}_1$ is the best constant in the Sobolev inequality for functions in $W^{1,1}_0(\Omega)$) \bk  and $h>0$ to be chosen. 

	Now we gather \eqref{bounded2tris}, \eqref{bounded2bis} and \eqref{bounded2} into \eqref{bounded1}, obtaining that
	
	\begin{equation}\label{bounded2quat}
		\begin{aligned}
		&\inf_{|s|\in [k,\infty)} |g_p(s)| \int_\Omega |G_k(u_p)| +  \int_\Omega |\nabla G_k(u_p)| 
		\\
		&\le h\int_\Omega |G_k(u_p)| + \|f\chi_{\{|f|> h\}}\|_{L^N(\Omega)} \mathcal{S}_1 \int_{\Omega} |\nabla G_k(u_p)| +|\{|u_p|>k\}|.
		\end{aligned}
	\end{equation}
	
	Now  we fix $h$ large enough in order to have
\begin{equation}\label{hgrande}	
	\|f\chi_{\{|f|> h\}}\|_{L^N(\Omega)} \mathcal{S}_1<1\,.
\end{equation}
Choosing $p_0>1$ such that $\frac1{p_0-1}>h$ we can pick  $\overline{k}$ such that 
	$$\inf_{|s|\in[k,\infty)}|g_p(s)|\ge h, \ \ \forall k\geq \overline{k}, $$
	for any $1<p<p_0$.
 From now on, we only consider those $p$ satisfying $1<p<p_0$.

	This allows to deduce from \eqref{bounded2quat} that it holds
	\begin{equation*}\label{bounded3}
		 \int_\Omega |\nabla G_k(u_p)| \le \frac{|\{|u_p|>k\}|}{1-\|f\chi_{\{|f|> h\}}\|_{L^N(\Omega)}\mathcal{S}_1},\ \  \ \ \forall k\geq \overline{k}.
	\end{equation*}

	An application of the Sobolev inequality gives that ($\ell>k$)
	\begin{equation*}\label{bounded7}
		|\ell-k| |\{|u_p|> \ell\}|^{\frac{N-1}{N}}\le \left(\int_\Omega |G_k(u_p)|^{\frac{N}{N-1}} \right)^{\frac{N-1}{N}}  \le \frac{\mathcal{S}_1|\{|u_p|>k\}|}{1-\|f\chi_{\{|f|> h\}}\|_{L^N(\Omega)}\mathcal{S}_1},  \ \  \ \ \forall \ell>k\geq \overline{k},
	\end{equation*}
	which is
	\begin{equation}\label{bounded8}
		|\{|u_p|> \ell\}|\le \frac{\mathcal{S}_1^{\frac{N}{N-1}}|\{|u_p|> k\}|^{\frac{N}{N-1}}}{\left(\left(1-\|f\chi_{\{|f|> k\}}\|_{L^N(\Omega)} \mathcal{S}_1\right) |\ell-k|\right)^{\frac{N}{N-1}}} \ \  \ \ \forall \ell>k\geq \overline{k}.
	\end{equation}
	Estimate \eqref{bounded8} is sufficient in order  to apply standard Stampacchia machinery (see \cite{st})  to deduce that $u_p$ is uniformly bounded with respect to $p$. The proof is concluded.
\end{proof}

The previous lemma easily allows to show a $BV$-estimate for $u_p$.
 
\begin{lemma}\label{lemma_stimeLN}
	Let $f\in L^N(\Omega)$ and let $g$ satisfy \eqref{condg}. Let $u_p$ be a solution to \eqref{pbp}. Then $u_p$ is unifomly  bounded in $BV(\Omega)$ (with respect to $p$), and  it also holds
	\begin{equation}\label{stimatermineenergia}
		(p-1)\int_\Omega |\nabla u_p|^p\le C,
	\end{equation}
	for some constant $C$ independent of $p$.
\end{lemma}
\begin{proof}
	It is sufficient to pick $v=u_p$ as a test function in \eqref{pbpw} obtaining
	\begin{equation}\begin{aligned}\label{stimeaprioriN1}
			\int_\Omega g_p(u_p)u_p+ \int_\Omega \frac{|\nabla u_p|^2}{\sqrt{1+|\nabla u_p|^2}} +(p-1)\int_\Omega |\nabla u_p|^p =  \int_\Omega f_p u_p  \stackrel{\eqref{bounded}}{\le} C \int_\Omega |f|\,.
	\end{aligned}\end{equation}

Now observe that
	\begin{equation}\label{stimeaprioriN3}
		\int_\Omega \frac{|\nabla u_p|^2}{\sqrt{1+|\nabla u_p|^2}}  = \int_{\Omega} {\sqrt{1+|\nabla u_p|^2}}  - \int_{\Omega} \frac{1}{\sqrt{1+|\nabla u_p|^2}}  \ge \int_\Omega |\nabla u_p| - |\Omega|.
	\end{equation}
Therefore it follows from gathering \eqref{stimeaprioriN3} into \eqref{stimeaprioriN1} that 
$$
		\int_\Omega |\nabla u_p| +(p-1)\int_\Omega |\nabla u_p|^p  \le C,
$$

thanks also to \eqref{condg}. This concludes the proof.
\end{proof}
From Lemmas \ref{lem_bounded} and \ref{lemma_stimeLN} we immediately deduce the following corollary.
\begin{corollary}\label{cor_u}
		Let $f\in L^N(\Omega)$ and let $g$ satisfy \eqref{condg}. Let $u_p$ be a solution to \eqref{pbp}. Then there exists $u\in BV(\Omega)\cap L^\infty(\Omega)$ such that, up to subsequences, $u_p$ strongly converges to $u$ in $L^q(\Omega)$ for every $q<\infty$  as $p\to 1^+$.
\end{corollary}

From now on, when referring to $u$ we mean the function found in Corollary \ref{cor_u}.

\begin{lemma}\label{lemma_esistenzazL2}
		Let $f\in L^N(\Omega)$ and let $g$ satisfy \eqref{condg}. Then there exists $z\in X(\Omega)_N$ such that
	\begin{equation}\label{esistenzaz1}
		g(u)-\operatorname{div}z= f  \ \ \ \text{in  } \mathcal{D'}(\Omega),
	\end{equation}
	and
	\begin{equation}\label{esistenzaz2}
		(z,Du)=\sqrt{1+|Du|^2} - \sqrt{1-|z|^2} \ \ \text{as measures in } \Omega.
	\end{equation}

	Furthermore,
	\begin{equation}\label{equint}
		\int_{\{|u|\ge k\}} |g(u)| \le \int_{\{|u| \ge k\}} |f|,
	\end{equation}
	holds for every $k>0$. 
\end{lemma}
\begin{proof}
Let $u_p$ be the solution of \eqref{pbp}. 	Firstly observe that, since ${|\nabla u_p|}{({1+|\nabla  u_p|^2})^{-\frac12}}\le 1$, there exists $z\in L^\infty(\Omega)^N$ such that ${\nabla u_p}{({1+|\nabla  u_p|^2})^{-\frac12}}$ converges to $z$ weak$^*$ in $L^\infty(\Omega)^N$ as $p\to 1^+$. It also follows from the weak lower semicontinuity of the norm that $||z||_{L^\infty(\Omega)^N} \le 1$.
	
	Moreover, the above argument, Lemmas \ref{lem_bounded} and \ref{lemma_stimeLN},  and Corollary \ref{cor_u}  give that \eqref{esistenzaz1} holds true. Indeed, we only need to show that the third term in \eqref{pbp} goes to zero in the sense of distributions as $p\to 1^+$; to do that, consider $\varphi \in C^1_c(\Omega)$ and observe that from the H\"older inequality and from \eqref{stimatermineenergia}, one has 
	\begin{equation}\label{termineenergiazero}
		\begin{aligned}
			(p-1)\left|\int_\Omega |\nabla u_p|^{p-2}\nabla u_p\cdot\nabla \varphi \right| &\le (p-1) \left(\int_\Omega |\nabla u_p|^p\right)^{\frac{p-1}{p}} \left(\int_\Omega |\nabla \varphi|^p\right)^{\frac{1}{p}}
			\\
			&\le (p-1)^{\frac{1}{p}} ||\nabla \varphi||_{L^\infty(\Omega)^N}|\Omega|^{\frac{1}{p}} \left((p-1)\int_\Omega |\nabla u_p|^p\right)^{\frac{p-1}{p}}
			\\
			&\stackrel{\eqref{stimatermineenergia}}{\le}(p-1)^{\frac{1}{p}} ||\nabla \varphi||_{L^\infty(\Omega)^N}|\Omega|^{\frac{1}{p}} C^{\frac{p-1}{p}},
		\end{aligned}
	\end{equation}
	which gives that
	$$\lim_{p\to 1^+}(p-1)\int_\Omega |\nabla u_p|^{p-2}\nabla u_p\cdot\nabla \varphi =0.$$
	
	This implies  \eqref{esistenzaz1} and, in particular,  that $z\in X(\Omega)_N$ as $g(u)\in L^\infty(\Omega)$.
	
	\medskip
	
	Let us also underline, for later purposes,  that,  since $u\in L^\infty(\Omega)$ and $f\in L^N(\Omega)$, then it holds
	\begin{equation}\label{moltperu}
	-u\operatorname{div}z =  (f-g(u)) u,
	\end{equation}
	almost everywhere in $\Omega$.

	\medskip
	
	Now we have to show \eqref{esistenzaz2} which  consists (recall Remark \ref{remdef}) in  proving both 
		\begin{equation}\label{zducontinua}
		z\cdot D^a u = \sqrt{1+ |D^a u|^2} + \sqrt{1-|z|^2}
	\end{equation}
	and
	\begin{equation} \label{zdusingolare}
		(z,Du)^s = |D^s u|.
	\end{equation}
	
	\medskip
	
	Hence,  let $0\le \varphi \in C^1_c(\Omega)$ and consider $v=u_p\varphi$ in  \eqref{pbpw}; this takes to
	
	\begin{equation}\begin{aligned}\label{lemmaz1}	
			& \int_\Omega g_p(u_p)u_p\varphi +\int_{\Omega} \frac{|\nabla u_p|^{2}\varphi}{\sqrt{1 +|\nabla u_p|^2}} + \int_{\Omega} \frac{\nabla u_p \cdot \nabla \varphi u_p}{\sqrt{1 +|\nabla u_p|^2}} + (p-1)\int_\Omega |\nabla u_p|^p\varphi
			\\
			& + (p-1)\int_\Omega |\nabla u_p|^{p-2}\nabla u_p\cdot \nabla \varphi u_p = \int_{\Omega}  f_p u_p \varphi.	
		\end{aligned}
	\end{equation}
	Dropping the nonnegative fourth term in \eqref{lemmaz1}, one gets
	\begin{equation}\begin{aligned}\label{lemmaz2}	
			&\int_\Omega g_p(u_p)u_p\varphi + \int_{\Omega} \sqrt{1 +|\nabla u_p|^2}\varphi - \int_{\Omega} \sqrt{1-\frac{|\nabla u_p|^{2}}{1 +|\nabla u_p|^2}}\varphi + \int_{\Omega} \frac{\nabla u_p \cdot \nabla \varphi u_p}{\sqrt{1 +|\nabla u_p|^2}}
			\\
			&+  (p-1)\int_\Omega |\nabla u_p|^{p-2}\nabla u_p\cdot \nabla \varphi u_p \le \int_{\Omega}  f_p u_p \varphi,	
	\end{aligned}\end{equation}
	where we used that	
	\begin{equation}
		\label{firstterm}
		\int_{\Omega} \frac{|\nabla u_p|^{2}\varphi}{\sqrt{1 +|\nabla u_p|^2}} = \int_{\Omega} \sqrt{1 +|\nabla u_p|^2}\varphi - \int_{\Omega} \sqrt{1-\frac{|\nabla u_p|^{2}}{1 +|\nabla u_p|^2}}\varphi.
	\end{equation}
	Now we aim to take the liminf as $p\to 1^+$ in \eqref{lemmaz2}. As $u_p$ strongly converges to $u$ in $L^q(\Omega)$ for any $q<\infty$ and $f_p$ strongly converges to $f$ in $L^N(\Omega)$ as $p\to 1^+$, we can easily pass to the limit in the first and in the last term of \eqref{lemmaz2}. The second term on the left-hand side of \eqref{lemmaz2} is lower semicontinuous with respect to the $L^1$ convergence. The nonpositive third term on the left-hand side of \eqref{lemmaz2} is weakly lower semicontinuous with respect to the $L^1$ convergence (recall \eqref{j2}).	Concerning  the fourth  term on the left-hand side of \eqref{lemmaz2} we use the weak$^*$ convergence of $\nabla u_p(1 +|\nabla u_p|^2)^{-\frac{1}{2}}$ to $z$ in $L^\infty(\Omega)^N$ as well as the strong convergence of $u_p$ in $L^q(\Omega)$ for any $q<\infty$ as $p\to 1^+$.
	
	Let us finally focus on the  last term on the left-hand side of \eqref{lemmaz2} for which   we reason as for \eqref{termineenergiazero}. Indeed one can apply the H\"older inequality with indexes $\left(\frac{p}{p-1}, p\right)$ obtaining that

			\begin{equation*}\label{termineenergiazero2}
			\begin{aligned}
				(p-1)\left|\int_\Omega |\nabla u_p|^{p-2}\nabla u_p\cdot\nabla \varphi u_p \right| &\le \|u_p\|_{L^\infty(\Omega)}(p-1) \left(\int_\Omega |\nabla u_p|^p\right)^{\frac{p-1}{p}} \left(\int_\Omega |\nabla \varphi|^{p}\right)^{\frac{1}{p}} 
				\\
				&\le \|u_p\|_{L^\infty(\Omega)} (p-1)^{\frac{1}{p}} ||\nabla \varphi||_{L^\infty(\Omega)^N}|\Omega|^{\frac{1}{p}} \left((p-1)\int_\Omega |\nabla u_p|^p\right)^{\frac{p-1}{p}}
			\end{aligned}
		\end{equation*}
	whose right-hand goes to zero as $p\to1^+$ thanks to \eqref{stimatermineenergia} and since $u_p$ is uniformly bounded with respect to $p$.

	Then we have proved that
	\begin{equation*}	
	\begin{aligned}	
		\int_\Omega g(u)u\varphi + \int_{\Omega} \sqrt{1 +|D u|^2}\varphi - \int_\Omega \sqrt{1-|z|^2}\varphi &\le - \int_{\Omega} uz\cdot \nabla \varphi + \int_{\Omega}  f u \varphi
		\\
		&\overset{\eqref{moltperu}}{=} -\int_{\Omega} uz\cdot \nabla \varphi-\int_\Omega u\operatorname{div}z \varphi + \int_\Omega g(u)u\varphi.	
	\end{aligned}
	\end{equation*}
	Hence, using  \eqref{dist1}, it holds
	\begin{equation}\label{zdu0}	
		\int_{\Omega} \sqrt{1 +|D u|^2}\varphi - \int_\Omega \sqrt{1-|z|^2}\varphi \le \int_\Omega (z, Du)\varphi,	\ \ \forall  \varphi \in C^1_c(\Omega), \ \varphi \ge 0.
	\end{equation}
	Now observe that, since $\operatorname{div}z \in L^N(\Omega)$ and $u\in L^\infty(\Omega)$, one can apply Lemma \ref{lemanzas} which allows to deduce from inequality \eqref{zdu0} that
	\begin{equation*}\label{zdu1}
		z\cdot D^a u \ge \sqrt{1 +|D^a u|^2} - \sqrt{1-|z|^2},
	\end{equation*}
	almost everywhere in $\Omega$.
	This gives \eqref{zducontinua} since the reverse inequality follows by \eqref{ine:2}.
	
	Proving \eqref{zdusingolare} is immediate by observing that $||z||_{L^\infty(\Omega)^N}\le 1$ implies
	$$(z,Du)^s \le |D u|^s= |D^s u|,$$
	as measures in $\Omega$. The reverse inequality follows by \eqref{zdu0}.
	
	Finally let us check the validity of  \eqref{equint};	we define the function ($k\ge\delta>0$):
	
	$$
	S_{\delta,k}(s)=
	\begin{cases}
		\sgn(s) &\text{if} \ \ |s|> k,
		\\
		0 &\text{if} \ \ |s|\le k-\delta,
		\\
		\displaystyle \frac{\sgn(s) (|s|-k+\delta)}{\delta} \ \ &\text{if} \ \ k-\delta< |s|\le k,
	\end{cases}
	$$
	and we take $v=S_{\delta,k}(u_p)$ in \eqref{pbpw} yielding to
	\begin{equation*}
		\int_\Omega g_p(u_p) S_{\delta,k}(u_p) \le \int_\Omega f_pS_{\delta,k}(u_p)\le \int_{\{|u_p|> k-\delta\}} |f_p|,
	\end{equation*}
	getting rid of the nonnegative second and third term. Then taking the limsup first as  $p\to 1^+$ and then as $\delta \to 0^+$,  one obtains \eqref{equint}.
\end{proof}

\begin{lemma}\label{lemma_datoalbordoL2}
		Let $f\in L^N(\Omega)$ and let $g$ satisfy \eqref{condg}. Then it holds
	$$u(\sgn{u} + [z,\nu])(x)=0 \ \text{for  $\mathcal{H}^{N-1}$-a.e. } x \in \partial\Omega,$$
	where $u$ and $z$ are the function and the vector field found in Corollary \ref{cor_u} and in Lemma \ref{lemma_esistenzazL2}.	
\end{lemma}
\begin{proof}
	 Let $u_p$ be a solution to \eqref{pbp} and let us take $v=u_p$ in \eqref{pbpw} yielding to
	\begin{equation*}
		\int_\Omega g_p(u_p)u_p + \int_{\Omega} \frac{|\nabla u_p|^2}{\sqrt{1+|\nabla u_p|^2}} + \int_{\partial \Omega} |u_p| d\mathcal{H}^{N-1}  \le  \int_{\Omega}  f_pu_p,	
	\end{equation*}		
	since  $u_p$ has zero Sobolev trace.

	Moreover reasoning as for \eqref {firstterm} (with $\varphi =1$) one obtains
	\begin{equation*}\label{boundary0}
		\int_\Omega g_p(u_p)u_p +\int_{\Omega} \sqrt{1+|\nabla u_p|^2} - \int_{\Omega} \sqrt{1-\frac{|\nabla u_p|^2}{1+|\nabla u_p|^2}} + \int_{\partial \Omega}|u_p| d\mathcal{H}^{N-1} \le \int_{\Omega}  f_p u_p.	
	\end{equation*}
	
	Now we can take the liminf as $p\to 1^+$ acting similarly to what done in Lemma \ref{lemma_esistenzazL2}. This allows to deduce that
	\begin{equation*}\label{boundary1}
		\int_\Omega g(u)u + \int_{\Omega} \sqrt{1+|D u|^2} - \int_{\Omega} \sqrt{1-|z|^2} + \int_{\partial \Omega}|u| d\mathcal{H}^{N-1} \le  \int_{\Omega}  fu.	
	\end{equation*}
	Now, recalling \eqref{moltperu} and \eqref{green}, one can write 
	\begin{equation*}\label{bordo1}
		\begin{aligned}
			\int_\Omega g(u)u + \int_{\Omega} \sqrt{1+|Du|^2} - \int_{\Omega} \sqrt{1-|z|^2} + \int_{\partial \Omega}|u| d\mathcal{H}^{N-1}  & \stackrel{\eqref{moltperu}}{\le}  -\int_\Omega u\operatorname{div}z +\int_\Omega g(u)u
			\\
			&\stackrel{\eqref{green}}{=} \int_{\Omega}(z,Du) - \int_{\partial \Omega} u[ z,\nu]d\mathcal{H}^{N-1} + \int_\Omega g(u)u.
		\end{aligned}	
	\end{equation*}
	Then the proof of Lemma \ref{lemma_datoalbordoL2} follows by  \eqref{esistenzaz2} and by the  fact that $|[z,\nu]|\leq 1$ $\mathcal{H}^{N-1}-$a.e. on $\partial \Omega$.
\end{proof}

\begin{proof}[Proof of Theorem \ref{teomain}]
	Let $u_p$ be a solution to \eqref{pbp}. Then it follows from Lemmas \ref{lem_bounded} and \ref{lemma_stimeLN} that $u_p$ is bounded in $BV(\Omega)\cap L^\infty(\Omega)$ with respect to $p$. Corollary \ref{cor_u} guarantees that $u_p$ converges, up to subsequences, to $u$ in $L^q(\Omega)$ for every $q<\infty$.\bk 
	
	Then \eqref{def_distrp=1} and \eqref {def_zp=1} are proved in Lemma \ref{lemma_esistenzazL2}.
	 Finally \eqref{def_bordop=1} follows from Lemma \ref{lemma_datoalbordoL2}. The proof is concluded.
\end{proof}

\subsection{Uniqueness of finite energy solutions}

In this section we prove Theorem \ref{teomainuniqueL2}. Let us explicitly highlight that our proof of the uniqueness result is strongly related to the presence of the absorption term. \bk

\begin{proof}[Proof of Theorem \ref{teomainuniqueL2}]
	Let $u_1$ and $u_2$ be solutions to \eqref{pb} and let $z_1$ and $z_2$ be the corresponding vector fields. Using  \eqref{def_distrp=1} (recall Remark \ref{remdef}), we readily have 
	\begin{equation}\label{testL2}
		\int_\Omega g(u_i) v -\int_\Omega v\,\operatorname{div} z_i = \int_\Omega  fv, \ \ \forall v\in BV(\Omega), \ \ i=1,2.
	\end{equation}	
	We take $v=u_1 - u_2$ in the difference between  two weak formulations \eqref{testL2} related to $u_1$ and $u_2$, obtaining
	
	\begin{equation}\label{unique1L2}
		\begin{aligned}
			&\int_\Omega (g(u_1)-g(u_2))(u_1-u_2) + \int_\Omega (z_1, Du_1)-\int_\Omega(z_2, Du_1)
			\\
			&+ \int_\Omega(z_2, Du_2) - \int_\Omega(z_1, Du_2) - \int_{\partial\Omega}(u_1-u_2)[z_1,\nu]\, d\mathcal H^{N-1}				    	
			\\
			&+\int_{\partial\Omega}(u_1-u_2)[z_2,\nu]\, d\mathcal H^{N-1}						    	
			= 0,
		\end{aligned}
	\end{equation}
	
	after an application of \eqref{green}.
	
	Observe first that from \eqref{def_bordop=1} it holds
	$$
	u_{i}(\sgn{u_i}+[z_{i},\nu])=0\ \ \mathcal H^{N-1}-\text{a.e. on}\ \partial\Omega\  \ \text{for}\ \  i=1,2.
	$$
	
	Hence one can rewrite the boundary terms as
	
	\begin{equation}\label{boundaryuniqueL2}
		\begin{aligned}
			&-\int_{\partial\Omega}(u_1-u_2)[z_1,\nu]\, d\mathcal H^{N-1}		+ \int_{\partial\Omega}(u_1-u_2)[z_2,\nu]\, d\mathcal H^{N-1}
			\\
			&= 	\int_{\partial\Omega}(|u_1|  + u_2[z_1,\nu])\, d\mathcal H^{N-1}	+ \int_{\partial\Omega}(|u_2| + u_1[z_2,\nu])\, d\mathcal H^{N-1}	
			\\
			&= 	\int_{\partial\Omega}(|u_1|  + u_1[z_2,\nu])\, d\mathcal H^{N-1}	+ \int_{\partial\Omega}(|u_2| + u_2[z_1,\nu])\, d\mathcal H^{N-1},	
		\end{aligned}
	\end{equation}
	which are nonnegative since $|[z_i,\nu]|\le 1$ for $i=1,2$.
	
	Gathering \eqref{boundaryuniqueL2} into \eqref{unique1L2} gives that
	
	\begin{equation*}
		\begin{aligned}
			\int_{\Omega} (g(u_1)-g(u_2))(u_1-u_2) &+ \int_\Omega (z_1, Du_1)-\int_\Omega(z_2, Du_1)
			\\
			&+ \int_\Omega(z_2, Du_2) - \int_\Omega(z_1, Du_2)  \le 0.
		\end{aligned}
	\end{equation*}
	
	Moreover, using \eqref{def_zp=1}, one gets
	
	\begin{equation*}
		\begin{aligned}
			&\int_{\Omega} (g(u_1)-g(u_2))(u_1-u_2) + \int_\Omega\sqrt{1+|Du_1|^2} - \int_\Omega \sqrt{1-|z_1|^2} - \int_\Omega(z_2, Du_1)
			\\
			&+ \int_\Omega \sqrt{1+|Du_2|^2} - \int_\Omega \sqrt{1-|z_2|^2} - \int_\Omega(z_1, Du_2)  \le 0.
		\end{aligned}
	\end{equation*}

	Now we aim to prove that
	$$\sqrt{1+|Du_1|^2} - \sqrt{1-|z_2|^2}  \ge (z_2,Du_1)$$
	and that
	$$\sqrt{1+|Du_2|^2} - \sqrt{1-|z_1|^2}  \ge (z_1,Du_2)\,,$$
	as measures in $\Omega$.
	This easily follows by splitting the measures in the absolutely continuous and singular parts.
	Let us observe that for the absolutely continuous part of the measures one needs that
	$$\sqrt{1+|D^a u_1|^2} - \sqrt{1-|z_2|^2}  \ge (z_2,Du_1)^a=z_2\cdot D^a u_1$$
	and that
	$$\sqrt{1+|D^a u_2|^2} - \sqrt{1-|z_1|^2}  \ge (z_1,Du_2)^a=z_1\cdot D^a u_2,$$
	which are given by  \eqref{ine:2} once one recalls that
	$$|z_i| = \frac{|D^a u_i|}{\sqrt{1+|D^a u_i|^2}}\leq 1\ \ \ i=1,2.$$
	For the singular part it is sufficient to recall that $||z_i||_{L^\infty(\Omega)^N}\le 1$.
	
	Hence we have shown that
	
	\begin{equation*}
		\begin{aligned}
			\int_{\Omega} (g(u_1)-g(u_2))(u_1-u_2)\le0,
		\end{aligned}
	\end{equation*}
	which concludes the proof as $g$ is increasing.
\end{proof}

\section{The  case of $L^1$ data}
\label{secL1}

Here we deal with \eqref{pb} in presence of a merely integrable datum $f$ and, once again, $g$ satisfying \eqref{condg}.

In this case one can not expect finite energy solutions. We specify how a weak solution of problem \eqref{pb} is meant in this case.
\begin{defin}
	\label{weakdefL1}
	Let $f\in L^1(\Omega)$. A function $u$ which is almost everywhere finite in $\Omega$ and such that both $g(u)\in L^1(\Omega)$ and $T_k(u)\in BV(\Omega)$ for any $k>0$,  is a solution to problem \eqref{pb} if there exists $z\in X(\Omega)_1$ with $||z||_{L^\infty(\Omega)^N}\le 1$ such that
	\begin{align}
		&g(u)-\operatorname{div}z = f \ \ \text{in}\ \ \mathcal{D}'(\Omega), \label{def_distrp=1L1}
		\\
		&(z,DT_k(u))=\sqrt{1+|DT_k(u)|^2} - \sqrt{1-|z_k|^2} \label{def_zp=1L1} \ \text{as measures in } \Omega \ \text{with} \ z_k:=z\chi_{\{|u|\le k\}},
		\\
		&T_k(u)(\sgn{T_k(u)} + [z,\nu])(x)=0 \label{def_bordop=1L1} \ \text{for  $\mathcal{H}^{N-1}$-a.e. } x \in \partial\Omega,
	\end{align}
	for any $k>0$.
\end{defin}

\begin{remark}\label{remarkdefinizioni}
	Let us explicitly observe that a function $u$, solution to \eqref{pb} in the sense of Definition \ref{weakdef}, is also a solution to the same problem in the sense of Definition \ref{weakdefL1}. Indeed, if $z_k=z\chi_{\{|u|\le k\}}$, it follows from \eqref{zesplicito} that
	\begin{equation}\label{1remzL1}
		z \cdot D^aT_k(u)=\sqrt{1+|D^aT_k(u)|^2} - \sqrt{1-|z_k|^2}.
	\end{equation}
	Moreover, since $(z,DT_k(u))\le |DT_k(u)|$ one has that $(z,DT_k(u))^s\le |D^sT_k(u)|$. For the reverse inequality it is sufficient to observe that
	\[
	|D^sT_k(u)|+|D^sG_k(u)|=|D^su | =(z,Du)^s=(z, DT_k(u))^s+(z,DG_k(u))^s \le |D^sT_k(u)|+|D^sG_k(u)|.
	\]
	Then the above becomes an equality, which yields to 
	\begin{equation}\label{2remzL1}
		(z,DT_k(u))^s= |D^s T_k(u)|.
	\end{equation}
	Equations \eqref{1remzL1} and \eqref{2remzL1} show that \eqref{def_zp=1L1} holds. This is sufficient to conclude that $u$ is a solution to \eqref{pb} in the sense of Definition \ref{weakdefL1}. We stress that conditions \eqref{def_zp=1L1} and \eqref{def_bordop=1L1} are the translation of, resp.,  \eqref{def_zp=1} and \eqref{def_bordop=1} to the $L^1$--setting and they formally tend to them as $k\to+\infty$. 
\end{remark}

\begin{theorem}\label{teoL1}
	Let $f\in L^1(\Omega)$ and let $g$ satisfy \eqref{condg}. Then there exists a solution to problem \eqref{pb} in the sense of Definition \ref{weakdefL1}.
\end{theorem}

We also state the following uniqueness result:
\begin{theorem}\label{teomainunique}
Let $g$ be an increasing function. Then there is at most one solution to problem \eqref{pb} in the sense of Definition \ref{weakdefL1}.
\end{theorem}

\subsection{Existence of infinite energy solutions}

By exploiting the results of Section \ref{secL2},  we work by approximation via the following problems
\begin{equation}
	\label{pbL1approx}
	\begin{cases}
		g(u_n)\dis -\operatorname{div}\left(\frac{D u_n}{\sqrt{1+|D u_n|^2}}\right) = f_n & \text{in}\;\Omega,\\
		u_n=0 & \text{on}\;\partial\Omega,
	\end{cases}
\end{equation}
where $f_n:= T_n(f)$. The existence of a solution  $u_n\in BV(\Omega)\cap L^{\infty}(\Omega)$ is proved in Theorem \ref{teomain}. This means that there exists $z_n\in X(\Omega)_N$ with $||z_n||_{L^\infty(\Omega)^N}\le 1$ such that it holds
\begin{align}
	&	\int_\Omega g(u_n) v -\int_\Omega v\,\operatorname{div} z_n = \int_\Omega  f_nv, \ \ \forall v\in BV(\Omega),\ \  \label{def_distrp=1approx}
	\\
	&(z_n,Du_n)=\sqrt{1+|Du_n|^2} - \sqrt{1-|z_n|^2} \label{def_zp=1approx} \ \ \ \ \text{as measures in } \Omega,
	\\
	&u_n(\sgn{u_n} + [z_n,\nu])(x)=0 \label{def_bordop=1approx}\ \ \ \text{for  $\mathcal{H}^{N-1}$-a.e. } x \in \partial\Omega.
\end{align}

We begin by proving estimates in $BV(\Omega)$ with respect to $n$ for any truncation of the approximating solutions.

\begin{lemma}\label{lemma_stimeL1}
	Let $f\in L^1(\Omega)$ and let $g$ satisfy \eqref{condg}. Let $u_n$ be the solution to \eqref{pbL1approx} given by Theorem \ref{teomain}. Then $$\|T_k(u_n)\|_{BV(\Omega)}\leq C (k +1),\ \ \forall k>0, $$ where $C$ is a positive constant not depending on $n$. Then there exists an almost everywhere finite function $u$ such that $T_k(u)\in BV(\Omega)$ for any $k>0$. Moreover, up to subsequences, $u_n\to u$ a.e. on $\Omega$, and  $g(u_n)\to g(u)$ in $L^1(\Omega)$ as $n\to\infty$.
\end{lemma}
\begin{proof}
Let $k>0$ and let us fix $v=T_k(u_n)$ in \eqref{def_distrp=1approx} obtaining
\begin{equation*}
	\int_\Omega g(u_n) T_k(u_n) - \int_\Omega T_k(u_n)\,\operatorname{div} z_n \le k \|f\|_{L^1(\Omega)}.
\end{equation*}
Then recalling \eqref{green}, \eqref{def_bordop=1approx} and the fact that $u_n\in BV(\Omega)$, one gets
\begin{equation*}
	\int_\Omega g(u_n) T_k(u_n) + \int_\Omega (z_n, D T_k(u_n)) + \int_{\partial\Omega} |T_k(u_n)| \ d \mathcal{H}^{N-1} \le k \|f\|_{L^1(\Omega)}.
\end{equation*}
Now, recalling \eqref{def_zp=1approx} and the discussion in Remark \ref{remarkdefinizioni}, one has
\begin{equation*}\label{stimaL1}
	\int_\Omega g(u_n) T_k(u_n) + \int_\Omega \sqrt{1 + |D T_k(u_n)|^2} - \int_\Omega \sqrt{1-|z_n|^2\chi_{\{|u_n|\le k\}}} + \int_{\partial\Omega} |T_k(u_n)| \ d \mathcal{H}^{N-1} \le k \|f\|_{L^1(\Omega)},
\end{equation*}
which  readily implies \bk  that $T_k(u_n)$ is bounded in $BV(\Omega)$ with respect to $n$ for any $k>0$. This is sufficient to deduce the existence of a limit function $u$ to whom $u_n$ converges, up to subsequences, almost everywhere in $\Omega$ as $n\to\infty$. Moreover $T_k(u) \in BV(\Omega)$. 

It remains to show that $g(u_n)$, up to subsequences, converges to $g(u)$ in $L^1(\Omega)$ as $n\to\infty$.
First observe that an application of the Fatou Lemma with respect to $n$ in \eqref{equint} gives that $g(u)\in L^1(\Omega)$; observe, in particular,  that  it  implies that $u$ is almost everywhere finite in $\Omega$.
Hence, one has that  $\forall \varepsilon>0$  there exists $h$ such that $|\{|g(u_n)|\geq h\}|< \varepsilon$. Using the first assumption in \eqref{condg}  there exists a increasing sequence $k_h>0$ such that $\{|u_n|\geq k_h\}\subseteq \{|g(u_n)|\geq h\}$. Therefore,  the equi-integrability of $g(u_n)$ is a consequence of \eqref{equint} with $k_h$ in place of $k$.   The proof is concluded.
\end{proof}

\begin{lemma}\label{lemma_esistenzaL1}
	Let $f\in L^1(\Omega)$ and let $g$ satisfy \eqref{condg}. Then there exists $z\in X(\Omega)_1$ such that
	\begin{equation}\label{esistenzaz1L1}
		g(u)-\operatorname{div}z= f  \ \ \ \text{in  } \mathcal{D'}(\Omega),
	\end{equation}
	and
	\begin{equation}\label{esistenzaz2L1}
		(z,DT_k(u))=\sqrt{1+|DT_k(u)|^2} - \sqrt{1-|z_k|^2} \ \ \text{as measures in } \Omega \ \text{and for any }k>0,
	\end{equation}
	where $u$ is the function found in Lemma \ref{lemma_stimeL1} and $z_k=z\chi_{\{|u|\le k\}}$. 
\end{lemma}
\begin{proof}
Let $u_n$ be the solution to \eqref{pbL1approx} given by Theorem \ref{teomain} with vector field $z_n$ such that $|z_n|\le 1$. Then there exists $z\in L^\infty(\Omega)^N$ such that $z_n$ converges to $z$ weak$^*$ in $L^\infty(\Omega)^N$ as $n\to\infty$ and such that $||z||_{L^\infty(\Omega)^N} \le 1$. Then, recalling that from Lemma \ref{lemma_stimeL1} $g(u_n)$ converges, up to subsequences, to $g(u)$ in $L^1(\Omega)$ as $n\to\infty$, it is easy to prove that \eqref{esistenzaz1L1} holds since $f_n$ converges to $f$	in $L^1(\Omega)$.

Now in order to prove \eqref{esistenzaz2L1} one can take $v=T_k(u_n)\varphi$ ($k>0$ and $0\le\varphi\in C^1_c(\Omega)$) in \eqref{def_distrp=1approx} getting to

\begin{equation*}
\int_\Omega g(u_n) T_k(u_n)\varphi -\int_\Omega T_k(u_n)\varphi\operatorname{div} z_n = \int_\Omega  f_nT_k(u_n)\varphi,
\end{equation*}
that, using \eqref{dist1}, gives

\begin{equation}\label{campozTK}
	\int_\Omega (z_n,D T_k(u_n))\varphi =  \int_\Omega  (f_n-g(u_n))T_k(u_n)\varphi - \int_\Omega z_n\cdot\nabla \varphi T_k(u_n).
\end{equation}
Then, recalling  Remark \ref{remarkdefinizioni}, one has  that,  for every $k>0$
	\begin{equation}\label{esistenzaz2TK}
 (z_n,DT_k(u_n))=\sqrt{1+|DT_k(u_n)|^2} - \sqrt{1-|z_n|^2\chi_{\{|u_n|\le k\}}},\ \text{as measures in $\Omega$},
 \end{equation}
which gathered in \eqref{campozTK} yields to 
\begin{equation*}
	\int_\Omega \sqrt{1+|D T_k(u_n)|^2}\varphi -\int_\Omega \sqrt{1-|z_n|^2\chi_{\{|u_n|\le k\}}}\varphi =  \int_\Omega  (f_n-g(u_n))T_k(u_n)\varphi - \int_\Omega z_n\cdot\nabla \varphi T_k(u_n).
\end{equation*}
Now one can let $n$ go to $\infty$ in the previous identity recalling that the left-hand is lower semicontinuous as already shown in the proof of Lemma \ref{lemma_esistenzazL2}. In particular, for the second term on the left-hand one uses that $z_n\chi_{\{|u_n|\le k\}}$ converges to $z\chi_{\{|u|\le k\}}$ weakly in $L^1(\Omega)^N$ as $n\to\infty$, for almost every $k>0$. 
Moreover the first term on the right-hand simply passes to the limit since $f_n$,$g(u_n)$ converge in $L^1(\Omega)$ and $T_k(u_n)$ converges weak$^*$ in $L^\infty(\Omega)$. Finally, the last term easily pass to the limit as  $T_k(u_n)$ converges in $L^1(\Omega)$  and $z_n$ converges weak$^*$ in $L^\infty(\Omega)^N$.

  This argument takes to (recall that $z_k:=z\chi_{\{|u|\le k\}}$)

\begin{equation*}
	\begin{aligned}
	\int_\Omega \sqrt{1+|D T_k(u)|^2}\varphi -\int_\Omega \sqrt{1-|z_k|^2}\varphi &\le  \int_\Omega  (f-g(u))T_k(u)\varphi - \int_\Omega z\cdot\nabla \varphi T_k(u)
	\\
	&= -\int_\Omega \operatorname{div}z T_k(u)\varphi - \int_\Omega z\cdot\nabla \varphi T_k(u)
	\\
	&=\int_\Omega (z,DT_k(u))\varphi,
	\end{aligned}
\end{equation*}
where the last passages follow from \eqref{esistenzaz1L1} and \eqref{dist1} respectively. From now on the reasoning to deduce \eqref{esistenzaz2L1} is similar to the one given in the proof of Lemma \ref{lemma_esistenzazL2}. Indeed it is sufficient to observe that $z\cdot D^aT_k(u) = z_k\cdot D^aT_k(u)$. This shows \eqref{esistenzaz2L1} for almost every $k>0$. Now observe that, reasoning as in Remark \ref{remdef}, from \eqref{esistenzaz2L1} one readily gets
\begin{equation}\label{zesplicitok}
	z_k= \frac{D^a T_k(u)}{\sqrt{1+|D^a T_k(u)|^2}},
\end{equation}
for almost every $k>0$. We claim that, for any fixed $k>0$, $z=0$ almost everywhere in $\{|u|=k\}$. If this is the case, then $z_n\chi_{\{|u_n|\le k\}}$ converges to $z\chi_{\{|u|\le k\}}$ weakly in $L^1(\Omega)^N$ as $n\to\infty$, for every $k>0$ and this concludes the proof. 
Let us finally check the claim; let us fix $h>k$ such that \eqref{zesplicitok} holds for $z_h=z\chi_{\{|u|\le h\}}$. Then, since $z_h=0$ almost everywhere in $\{|u|=k\}$, also $z=0$ almost everywhere on the same set.

\end{proof}
\begin{remark}\label{benilan}
Looking at \eqref{zesplicitok} one would like to identify $z$
 as
\begin{equation}\label{ben}
 z= \frac{D^a u}{\sqrt{1+|D^a u|^2}}\,.
 \end{equation}
 This is not accurate as we only ask for $u$ to have truncations in $BV(\Omega)$ so that $D^a u$ is not well defined in general. 
 
 Nevertheless,  reasoning as in \cite{b6} it is possible to define (see for instance \cite[Lemma 1]{ACM_MA}\bk) a generalized gradient for functions whose truncation is in  $BV(\Omega)$  for which\bk,  in turn, \eqref{ben} holds a.e. in $\Omega$. 
Indeed, let $u$ be a measurable function finite a.e. on $\Omega$ such that $T_k(u)\in BV(\Omega)$ for any $k>0$. Then $D^a T_k(u)$ is well defined for any $k>0$. A standard argument allows us to select a unique measurable vector function $v:\Omega\to\rn$ that satisfies 
$$
v_{\chi_{\{|u|\leq k\}}} = D^a T_k(u). 
$$

It is possible to show that, if $u\in BV(\Omega)$, then $v=D^a u$. 

Using this generalized gradient, the vector field $z$ given in Definition \ref{weakdefL1} can be uniquely identified by \eqref{ben}.  
\end{remark}

\begin{remark}\label{conv_zn}
	For subsequent use, we underline that in the previous proof we have shown that  $z_n\chi_{\{|u_n|\le k\}}$ converges to $z\chi_{\{|u|\le k\}}$ weakly in $L^1(\Omega)^N$ as $n\to\infty$ and for every $k>0$.
\end{remark}

\begin{lemma}\label{lemma_datoalbordoL1}
	Let $f\in L^1(\Omega)$ and let $g$ satisfy \eqref{condg}. Then it holds for any $k>0$
	$$T_k(u)(\sgn{T_k(u)} + [z,\nu])(x)=0 \ \text{for  $\mathcal{H}^{N-1}$-a.e. } x \in \partial\Omega,$$
	where $u$ and $z$ are the function and the vector field found in Lemma \ref{lemma_stimeL1} and in Lemma \ref{lemma_esistenzaL1}.	
\end{lemma}
\begin{proof}
	 Let $u_n$ be the solution of \eqref{pbL1approx} given by Theorem \ref{teomain}. Let us pick $v=T_k(u_n)$ in \eqref{pbL1approx} yielding to
	\begin{equation*}
		\int_\Omega g(u_n)T_k(u_n) - \int_{\Omega} \operatorname{div}z_n T_k(u_n) =  \int_{\Omega}  f_nT_k(u_n),	
	\end{equation*}		
	and, after an application of the \eqref{green}, to 
	\begin{equation}\label{bordoL1_1}
		\int_{\Omega} (z_n,DT_k(u_n)) - \int_{\partial\Omega} T_k(u_n)[z_n,\nu] =  \int_{\Omega}  (f_n-g(u_n))T_k(u_n).	
	\end{equation}	
	Now using both \eqref{esistenzaz2TK} and \eqref{def_bordop=1approx}, it follows from \eqref{bordoL1_1} that
	\begin{equation}\label{bordoL1_2}
	\int_{\Omega} \sqrt{1+|DT_k(u_n)|^2} - \int_\Omega \sqrt{1-|z_n|^2\chi_{\{|u_n|\le k\}}} + \int_{\partial\Omega} |T_k(u_n)| =  \int_{\Omega}  (f_n-g(u_n))T_k(u_n).	
	\end{equation}		
	Recalling also Remark \ref{conv_zn}, we can take $n\to\infty$ by lower semicontinuity of the left-hand of \eqref{bordoL1_2}. For the right-hand it is  sufficient to use  the strong convergence of both $f_n$ and $g(u_n)$ in $L^1(\Omega)$ and the $*$-weak convergence in $L^\infty(\Omega)$ of $T_k(u_n)$ as $n\to\infty$. Then one deduces
	
	\begin{equation}\label{bordoL1_3}
		\int_{\Omega} \sqrt{1+|DT_k(u)|^2} - \int_\Omega \sqrt{1-|z_k|^2} + \int_{\partial\Omega} |T_k(u)| \le  \int_{\Omega}  (f-g(u))T_k(u).	
	\end{equation}
	Now observe that from \eqref{esistenzaz1L1} one has that $(f-u)T_k(u)=-T_k(u)\operatorname{div}z$. Then an application of \eqref{green} in \eqref{bordoL1_3} gives
	\begin{equation*}
	\int_{\Omega} \sqrt{1+|DT_k(u)|^2} - \int_\Omega \sqrt{1-|z_k|^2} + \int_{\partial\Omega} |T_k(u)| \le   \int_{\Omega}  (z,DT_k(u)) - \int_{\partial\Omega} T_k(u)[z,\nu],	
	\end{equation*}	
	which, from \eqref{esistenzaz2L1}, implies
	\begin{equation*}
		\int_{\partial\Omega} |T_k(u)| + \int_{\partial\Omega} T_k(u)[z,\nu] \le 0,	
	\end{equation*}
	and this concludes the proof since $|[z,\nu]|\le 1$ $\mathcal{H}^{N-1}-$a.e. on $\partial \Omega$.
\end{proof}

\begin{proof}[Proof of Theorem \ref{teoL1}]
	Let $u_n$ be the solution to \eqref{pbL1approx} given by Theorem \ref{teomain}. It follows from Lemma \ref{lemma_stimeL1} that $T_k(u_n)$ is bounded in $BV(\Omega)$ with respect to $n$ and for any $k>0$. In the same lemma it is shown that $u_n$ converges, up to subsequences, as $n\to\infty$ almost everywhere in $\Omega$ to a function $u$, which is almost everywhere finite. Moreover $g(u_n)$  converges to $g(u)$ in $L^1(\Omega)$ as $n\to\infty$.
	Requests \eqref{def_distrp=1L1} and \eqref {def_zp=1L1} are proved in Lemma \ref{lemma_esistenzaL1}. The boundary condition \eqref{def_bordop=1L1} is shown in Lemma \ref{lemma_datoalbordoL1}. The proof is concluded.
\end{proof}

\subsection{Uniqueness of infinite energy solutions}

In this section we prove the uniqueness Theorem \ref{teomainunique} by strictly following the lines of the proof of Theorem \ref{teomainuniqueL2}.

\begin{proof}[Proof of Theorem \ref{teomainunique}]
	Let $u_1$ and $u_2$ be solutions to \eqref{pb} and let $z_1$ and $z_2$ be the corresponding vector fields.
	
	Then one has that
	\begin{equation}\label{test}
		\int_\Omega g(u_i) v -\int_\Omega v\,\operatorname{div} z_i = \int_\Omega  fv, \ \ \forall v\in BV(\Omega)\cap L^\infty(\Omega), \ \ i=1,2.
	\end{equation}	
	Let us observe that the main difference with respect to the proof of Theorem \ref{teomainuniqueL2} relies on the fact that $u_1,u_2$ are not suitable test functions in \eqref{test} anymore.
	
	Hence we have to take $v=T_k(u_1) - T_k(u_2)$ in the difference between two weak formulations \eqref{test} related to $u_1$ and $u_2$, yielding to
	
	\begin{equation}\label{unique1}
		\begin{aligned}
			&\int_\Omega (g(u_1)-g(u_2))(T_k(u_1) - T_k(u_2)) + \int_\Omega (z_1, DT_k(u_1))-\int_\Omega(z_2, DT_k(u_1))
			\\
			&+ \int_\Omega(z_2, DT_k(u_2)) - \int_\Omega(z_1, DT_k(u_2)) - \int_{\partial\Omega}(T_k(u_1)-T_k(u_2))[z_1,\nu]\, d\mathcal H^{N-1}				    	
			\\
			&+\int_{\partial\Omega}(T_k(u_1)-T_k(u_2))[z_2,\nu]\, d\mathcal H^{N-1}						    	
			= 0,
		\end{aligned}
	\end{equation}
	
	where we also used \eqref{green}.
	
	From \eqref{def_bordop=1L1} one has
	$$
	T_k(u_{i})(\sgn{T_k(u_i)}+[z_{i},\nu])=0\ \ \mathcal H^{N-1}-\text{a.e. on}\ \partial\Omega\  \ \text{for}\ \  i=1,2.
	$$
	
This means that
	
	\begin{equation}\label{boundaryunique}
		\begin{aligned}
			&-\int_{\partial\Omega}(T_k(u_1)-T_k(u_2))[z_1,\nu]\, d\mathcal H^{N-1}		+ \int_{\partial\Omega}(T_k(u_1)-T_k(u_2))[z_2,\nu]\, d\mathcal H^{N-1}
			\\
			&= 	\int_{\partial\Omega}(|T_k(u_1)|  + T_k(u_2)[z_1,\nu])\, d\mathcal H^{N-1}	+ \int_{\partial\Omega}(|T_k(u_2)| + T_k(u_1)[z_2,\nu])\, d\mathcal H^{N-1}	
			\\
			&= 	\int_{\partial\Omega}(|T_k(u_1)|  + T_k(u_1)[z_2,\nu])\, d\mathcal H^{N-1}	+ \int_{\partial\Omega}(|T_k(u_2)| + T_k(u_2)[z_1,\nu])\, d\mathcal H^{N-1},	
		\end{aligned}
	\end{equation}
	which is nonnegative since $|[z_i,\nu]|\le 1$ for $i=1,2$.
	
	Gathering \eqref{boundaryunique} into \eqref{unique1}, it  holds that
	
	\begin{equation*}
		\begin{aligned}
			\int_{\{|u_1|\le k,|u_2|\le k\}} (g(u_1)-g(u_2))(u_1-u_2) &+ \int_\Omega (z_1, DT_k(u_1))-\int_\Omega(z_2, DT_k(u_1))
			\\
			&+ \int_\Omega(z_2, DT_k(u_2)) - \int_\Omega(z_1, DT_k(u_2))  \le 0.
		\end{aligned}
	\end{equation*}
	
	Moreover, using \eqref{def_zp=1L1}, one gets ($z_{i,k}:= z_i\chi_{\{|u_i|\le k\}}$ for $i=1,2$)
		
		\begin{equation*}
			\begin{aligned}
				&\int_{\{|u_1|\le k,|u_2|\le k\}} (g(u_1)-g(u_2))(u_1-u_2) + \int_\Omega\sqrt{1+|DT_k(u_1)|^2} - \int_\Omega \sqrt{1-|z_{1,k}|^2} - \int_\Omega(z_2, DT_k(u_1))
				\\
				&+ \int_\Omega \sqrt{1+|DT_k(u_2)|^2} - \int_\Omega \sqrt{1-|z_{2,k}|^2} - \int_\Omega(z_1, DT_k(u_2))  \le 0.
			\end{aligned}
		\end{equation*}

		Now we  claim  that both
		$$\sqrt{1+|DT_k(u_1)|^2} - \sqrt{1-|z_{2,k}|^2}  \ge (z_2,DT_k(u_1))$$
		and 
		$$\sqrt{1+|DT_k(u_2)|^2} - \sqrt{1-|z_{1,k}|^2}  \ge (z_1,DT_k(u_2))\,$$
		hold as measures in $\Omega$.
		Once again, this follows by splitting it in the absolutely continuous and singular parts.
		For the absolutely continuous part of the measures one needs that
		$$\sqrt{1+|D^a T_k(u_1)|^2} - \sqrt{1-|z_{2,k}|^2}  \ge (z_2,DT_k(u_1))^a=z_2\cdot D^a T_k(u_1)$$
		and that
		$$\sqrt{1+|D^a T_k(u_2)|^2} - \sqrt{1-|z_{1,k}|^2}  \ge (z_1,DT_k(u_2))^a=z_1\cdot D^a T_k(u_2),$$
		which is inequality \eqref{ine:2} once one notices that $z_2\cdot D^a T_k(u_1)=z_{2,k}\cdot D^aT_k(u_1)$ and $z_1\cdot D^a T_k(u_2)=z_{1,k}\cdot D^aT_k(u_2)$.
	
	For the singular part it is sufficient to recall that $||z_i||_{L^\infty(\Omega)^N}\le 1$.
	
	This proves that
	
	\begin{equation*}
		\begin{aligned}
			\int_{\{|u_1|\le k,|u_2|\le k\}} (g(u_1)-g(u_2))(u_1-u_2)\le0,
		\end{aligned}
	\end{equation*}
for any $k>0$. The proof is concluded as $g$ is increasing.
\end{proof}

\section{Finite energy and unbounded solutions}
\label{boun}

So far we have shown the existence of a bounded $BV$-solution when the datum $f$ lies in $L^N(\Omega)$ (Theorem \ref{teomain}) while we proved the existence of an infinite energy solution when $f$ is merely integrable (Theorem \ref{teoL1}). 
One could wonder what happens to the solution's regularity when the datum $f$ is in between these two extreme cases.

\medskip

In particular, in the next result, we consider data lying in the Marcinkiewicz space $L^{N,\infty}(\Omega)$; again, let us refer to the monograph \cite{PKF} for an introduction and basic properties.
Among other things let us recall that for functions in $BV(\Omega)$ the natural embedding is in the Lorentz space $L^{\frac{N}{N-1},1}(\Omega)$ (see \cite{PKF}) where the best Sobolev constant is given by $\tilde{\mathcal S}_{1}= [(N-1)\omega_{N}^{\frac{1}{N}}]^{-1}$. 

\medskip

Let us state a first regularity result, in which we prove that a bounded solution exists provided the $L^{N,\infty}$-norm \bk of the datum is small enough. This result shows how the absorption given by a general $g$ only satisfying \eqref{condg} is to weak to infer boundedness for any data in the Marcinkiewicz space and it fits with the result in \cite{gop2}.

\begin{theorem}\label{teo_reg}
	Let $f\in L^{N,\infty}(\Omega)$ with $\|f\|_{L^{N,\infty}(\Omega)}<\tilde{\mathcal S}^{-1}_{1}$ and let $g$ satisfy \eqref{condg}. Then there exists a bounded solution $u$ to \eqref{pb} in the sense of Definition \ref{weakdef} (with $\operatorname{div} z$ belongs to $L^{N,\infty}(\Omega)$ in place of $L^N(\Omega)$).
\end{theorem}  
\begin{proof}
	The proof strictly follows the lines of the one of Theorem \ref{teomain} once that one uses the Sobolev inequality in $L^{\frac{N}{N-1},1}(\Omega)$.  
\end{proof}

\begin{remark} As a technical remark,  let us stress that, in the previous theorem, the smallness condition on the norm of $f$ is needed since we are not anymore able, in general, to fix $h$ great enough in order to deduce \eqref{hgrande} as in the proof of lemma \ref{lem_bounded} when $f\in L^{N,\infty}(\Omega)$. This is not only  a technical  obstruction, in fact in Example \ref{example} below we will show that unbounded solutions could exist if $f\in L^{N,\infty}(\Omega)$ with $\|f\|_{L^{N,\infty}(\Omega)}=\tilde{\mathcal S}^{-1}_{1}$ yielding optimality to the result of Theorem \ref{teo_reg}. 
 \end{remark} 
In the next regularity result we show how the existence of a finite energy solution  can be proven, also below the critical threshold $N$ for the datum,  provided some stronger growth assumption on $g$ is required. 

\begin{theorem}\label{teo_reg2}
	Let us assume that there exist $c_0>0$ and $q>1$ such that $g(s)s\ge c_0|s|^{q}$ for any $s\in \mathbb{R}$. Moreover let $f\in L^{q'}(\Omega)$. Then there exists a solution $u\in BV(\Omega)$ to \eqref{pb} in the sense of Definition \ref{weakdefL1} such that $u\in L^{q}(\Omega)$.
\end{theorem}  
\begin{proof}
	In Theorem \ref{teoL1} we have found the solution $u$ to \eqref{pb} as the almost everywhere limit (up to subsequences) in $n$ of $u_n\in BV(\Omega)\cap L^\infty(\Omega)$ solution to \eqref{pbL1approx}. Hence, to show that $u\in BV(\Omega)$, we only need to show that $u_n$ is bounded in $BV(\Omega)\cap L^{q}(\Omega)$ with respect to $n$.
	
	\medskip
	
	To this aim we take $u_n$ as a test function in \eqref{def_distrp=1approx}, yielding to
	\begin{equation*}
		\int_\Omega g(u_n) u_n - \int_\Omega u_n\,\operatorname{div} z_n =\int_\Omega f_n u_n,
	\end{equation*}
	which, recalling \eqref{green} and \eqref{def_bordop=1approx}, implies
	\begin{equation*}
		\int_\Omega g(u_n) u_n + \int_\Omega (z_n, D u_n) + \int_{\partial\Omega} |u_n| \ d \mathcal{H}^{N-1} = \int_\Omega f_n u_n.
	\end{equation*}
	Hence, from \eqref{def_zp=1approx}, one has
	\begin{equation}\label{stima_reg}
		\int_\Omega g(u_n) u_n + \int_\Omega \sqrt{1 + |D u_n|^2} - \int_\Omega \sqrt{1-|z_n|^2} + \int_{\partial\Omega} |u_n| \ d \mathcal{H}^{N-1}= \int_\Omega f_n u_n.
	\end{equation}
From \eqref{stima_reg} and after an application of the Young inequality, one obtains
	
	\begin{equation}\label{stima_reg3}
		\begin{aligned}
		&c_0\int_\Omega |u_n|^{q} + \int_\Omega \sqrt{1 + |D u_n|^2} - \int_\Omega \sqrt{1-|z_n|^2} + \int_{\partial\Omega} |u_n| \ d \mathcal{H}^{N-1} 
		\\
		&\le C\|f\|^{q'}_{L^{q'}(\Omega)} + \frac{c_0}{2}\int_\Omega |u_n|^{q}.
		\end{aligned}
	\end{equation}
	From \eqref{stima_reg3} it is simple to convince that $u_n$ is bounded in both $BV(\Omega)$ and $L^{q}(\Omega)$ with respect to $n$. The proof is concluded.
\end{proof}

Let us now summarize some of the results proven so far. If $f\in L^N(\Omega)$ then Theorem \ref{teomain} gives the existence of a bounded solution $u$ to \eqref{pb} in the sense of Definition \ref{weakdef}.

As we have just seen in Theorem \ref{teo_reg}, we can enlarge the set of admissible data to $L^{N,\infty}(\Omega)$ and still deducing existence of  bounded $BV$-solution to \eqref{pb} as long as we require a smallness condition on $f$.

Finally, in Theorem \ref{teo_reg2}, we proved that no smallness condition is required on $f$ to provide $BV$-solutions on condition that a suitable growth assumption is imposed on $g$. In this case, in general, the solutions are not bounded anymore. 

\medskip

Both of the results of this section are sharp as the next example shows.

\medskip

The following example shows that, for a suitable $f\in L^{N,\infty}(\Omega)$ with norm above the critical threshold found in Theorem \ref{teo_reg}, the unique solution to \eqref{pb} (here $g(s)=s$; see Theorem \ref{teomainuniqueL2}) is not bounded anymore. 

\begin{example}\label{example}
	Let us fix $\Omega = B_1(0)$, let $N\geq 3$ and let $0<\alpha < N-1$; it is not difficult to be  convinced that a radial solution to problem
	$$
	\label{pbie}
	\begin{cases}
		\dis -\operatorname{div}\left(\frac{D u_\alpha }{\sqrt{1+|D u_\alpha|^2}}\right) + u_\alpha = \frac{N-1}{|x|} \left(v_\alpha (|x|)+\frac{r^{-\alpha+1}-r}{N-1}\right)=:f_\alpha(x)& \text{in}\; B_1(0),\\
		u_\alpha=0 & \text{on}\;\partial B_1(0)\,
	\end{cases}
	$$
	is given, if $|x|=r$,  by $u_\alpha (x)=r^{-\alpha} - 1$ provided $v_\alpha:(0,1)\mapsto \mathbb{R}^+$ is given by
	
	$$
	v_{\alpha}(r)=\frac{\alpha r^{-\alpha-1}(\alpha^2 r^{-2\alpha -2}- \frac{\alpha+2 - N}{N-1})}{(1+\alpha^2 r^{-2\alpha-2})^{\frac{3}{2}}}\,.
	$$
	An explicit calculation of the norm gives that
	$$
	\left\|f_\alpha(x)\right\|_{L^{N,\infty}(B_1(0))}  = (N-1)\omega_{N}^{\frac{1}{N}} = \tilde{\mathcal S}_{1}^{-1};
	$$
 a similar computation can be found in Example $1$ of \cite{gop2}.
\end{example}

\section*{Acknowledgements}
F. Oliva and F. Petitta are partially supported by the Gruppo Nazionale per l’Analisi Matematica, la Probabilità e le loro Applicazioni (GNAMPA) of the Istituto Nazionale di Alta Matematica (INdAM). 
S. Segura de Le\'on is partially supported by  Grant PID2022-136589NB-I00 founded  by MCIN/AEI/10.13039/5
01100011033 as well as by Grant RED2022-134784-T funded by MCIN/AEI/1
0.13039/501100011033.

The authors would also like to warmly thank Prof. J.M. Maz\'on for bringing to our attention papers \cite{dt},  \cite{ACM_MA}, and \cite{gm} as well as  for further interesting comments on the preliminary version of this manuscript.


\begin{thebibliography}{10}	
\bibitem{afp} L. Ambrosio, N. Fusco and D. Pallara, Functions of Bounded Variation and Free Discontinuity Problems, Oxford Mathematical Monographs, 2000


\bibitem{ABCM} F. Andreu, C. Ballester, V. Caselles and J. M. Maz\'on,
The Dirichlet problem for the total variation flow, J. Funct. Anal.  180 (2), 347-403 (2001)

 \bibitem{ACM_RMI} F. Andreu, V. Caselles, and J.M. Maz\'on, A parabolic quasilinear problem for linear growth functionals, Rev. Mat. Iberoamericana 18 (2002), 135-185

 \bibitem{ACM_MA} F. Andreu, V. Caselles, and J.M. Maz\'on, Existence and uniqueness of solution for a parabolic quasi-
linear problem for linear growth functionals with L1 data, Math. Ann. 322 (2002), 139-206\bk

\bibitem{ACM} F. Andreu, V. Caselles and J. M. Maz\'on, Parabolic quasilinear equations minimizing linear growth functionals, Progress in
Mathematics, 223, Birkh\"auser Verlag, Basel, 2004

\bibitem{An} G. Anzellotti, Pairings between measures and bounded functions and compensated compactness, Ann. Mat. Pura Appl. 135 (4),  293-318 (1983)


\bibitem{b6} P. Benilan, L. Boccardo, T. Gallou\"{e}t, R.Gariepy, M. Pierre and J.L. Vazquez, An $L^1$ theory of existence and uniqueness of nonlinear elliptic equations,  Ann. Scuola Norm. Sup. Pisa 22 (1995) 240-273.
\bibitem{brezis} H. Brezis, Functional Analysis, Sobolev Spaces and Partial	Differential Equations,  Universitext. Springer, New York, 2011


\bibitem{CF} G. Q. Chen and H. Frid, Divergence-measure fields and hyperbolic conservation laws, Arch. Ration. Mech. Anal. 147 (2),  89-118 (1999)

\bibitem{CT} M. Cicalese and C. Trombetti, Asymptotic behaviour of solutions to $p$-Laplacian equation, Asymptotic Analysis
35, 27-40 (2003)

\bibitem{cf1} P. Concus and R. Finn, A singular solution of the capillary equation. I. Existence,	Invent. Math. 29 (2), 143-148 (1975)

\bibitem{cf2} P. Concus and R. Finn, A singular solution of the capillary equation. II. Uniqueness, Invent. Math. 29 (2), 149-159 (1975)


\bibitem{cl} M. G. Crandall and  T. M. Liggett,  Generation of semi-groups of nonlinear transformations on general Banach spaces. Amer. J. Math. 93 (1971), 265-298

\bibitem{DS} A. Dall'Aglio and S. Segura de Le\'on, Bounded solutions to the 1--Laplacian equation with a total variation term, Ric. Mat. 68 (2), 597-614 (2019)

\bibitem{DGOP} V. De Cicco, D. Giachetti, F. Oliva and  F. Petitta, The Dirichlet problem for singular elliptic equations with general nonlinearities, Calc. Var. Partial Differential Equations 58 (4), Paper No. 129 (2019)

\bibitem{dt} F. Demengel and R.  Temam, Convex functions of a measure and applications, Indiana Univ. Math. J. 33 (5), 673-709 (1984)

\bibitem{finn74} R. Finn, A note on the capillary problem, Acta Mathematica 132, 199-205 (1974)

\bibitem{fg84} R. Finn and E. Giusti, On nonparametric surfaces of constant mean curvature, Ann Sc. Norm Sup. Pisa 17, 389-403 (1984)


\bibitem{gop2} D. Giachetti, F. Oliva and F. Petitta, Bounded solutions for non-parametric mean curvature problems with nonlinear terms, arXiv: 2304.13611

\bibitem{GMP} L. Giacomelli, S. Moll and F.  Petitta,  Nonlinear diffusion in transparent media: the resolvent equation, Adv. Calc. Var. { 11}  (4), 405-432 (2018)

\bibitem{g} M. Giaquinta, On the Dirichlet problem for surfaces of prescribed mean curvature, Manuscripta Math. 12, 73-86 (1974)

\bibitem{g76} E. Giusti, Boundary value problems for non-parametric  surfaces of prescribed mean curvature, Ann. Scuola Norm. Sup. Pisa Cl. Sci. (4) 3 (3), 501-548 (1976)

\bibitem{g78} E. Giusti,  On the Equation of Surfaces of Prescribed Mean Curvature. Existence and uniqueness without
boundary conditions, Invent. math. 46, 111-137 (1978)

\bibitem{gm} W. G\'orny and J. M. Maz\'on,  A duality-based approach to gradient flows of linear growth functionals,  arXiv preprint arXiv:2212.08725


\bibitem{ww} D.L. Hu, M. Prakash, B. Chan et al., Water-walking devices, Exp Fluids 43, 769-778 (2007)

\bibitem{KS}  B.  Kawohl and F.  Schuricht, Dirichlet problems for the $1$-Laplace operator, including the eigenvalue problem, Commun. Contemp. Math. 9 (4), 515-543 (2007)


\bibitem{lops} M.  Latorre, F. Oliva and F. Petitta, S. Segura de Le\'on,  
The Dirichlet problem for the $1$-Laplacian with a general singular term and $L^1$-data, Nonlinearity 34 1791-1816 (2021)


\bibitem{lc}  G.P. Leonardi and G.E. Comi, The prescribed mean curvature measure equation in non-parametric form, arXiv:2302.10592 (2023)

\bibitem{ll} J. Leray and J. L. Lions, Quelques r\'esulatats de Vi\text{$\check{s}$}ik sur les probl\'emes elliptiques nonlin\'eaires par les m\'ethodes de Minty-Browder, Bull. Soc. Math. France 93, 97-107 (1965)


\bibitem{yy} YY. Liu, J. Xin and Y. Yu, Asymptotics for Turbulent Flame Speeds of the Viscous G-Equation Enhanced by Cellular and Shear Flows, Arch Rational Mech Anal 202, 461-492 (2011)

\bibitem{mst2} A. Mercaldo, S. Segura de Le\'on and C. Trombetti, On the solutions to $1$-Laplacian equation with $L^1$ data,   J. Funct. Anal. 256 (8), 2387-2416 (2009)


\bibitem{PKF}  L. Pick, A. Kufner, J. Oldrich and S. Fuc\'ik,  Function Spaces, 1. Berlin, Boston: De Gruyter, 2012


\bibitem{se} J. Serrin, The problem of Dirichlet for quasilinear elliptic equations with many independent variables, Philos. Trans. Roy. Soc. London Ser. A 264, 413-496 (1969)

\bibitem{st} G. Stampacchia, Le probl\`eme de Dirichlet pour les \'equations elliptiques du seconde ordre  \`a coefficientes discontinus,  Ann. Inst. Fourier (Grenoble) 15, 189-258  (1965)


\end{thebibliography}
\end{document}